\documentclass{amsart}

\usepackage{cancel}

\usepackage{datetime}
\usepackage{amsfonts}
\usepackage{amsmath}
\usepackage{amssymb}
\usepackage{amsthm}
\usepackage{enumerate}
\usepackage{eucal}
\usepackage{graphicx,color}
\usepackage[usenames,dvipsnames,svgnames,table]{xcolor}

\usepackage{hyperref}
\hypersetup{colorlinks,
            filecolor=black,
            linkcolor=MidnightBlue,
            citecolor=NavyBlue,
            urlcolor=RoyalBlue,
            bookmarksopen=true}

\swapnumbers
\theoremstyle{plain}
\newtheorem{theorem}{Theorem}
\newtheorem{lemma}[theorem]{Lemma}
\newtheorem{corollary}[theorem]{Corollary}
\newtheorem{prop}[theorem]{Proposition}
\newtheorem*{assumption}{Main Assumptions}
\newtheorem{definition}[theorem]{Definition}
\newtheorem{claim}[theorem]{Claim}
\newtheorem{remark}[theorem]{Remark}

\theoremstyle{remark}

\newcommand{\cv}{\nabla}

\newcommand{\form}{\mr{form}}
\newcommand{\init}{\mr{init}}

\newcommand{\heat}{\big(\partial_t-\Delta\big)}
\newcommand{\lp}{\langle}

\newcommand{\mb}{\mathbb}
\newcommand{\mc}{\mathcal}

\newcommand{\mr}{\mathrm}
\newcommand{\pd}{\partial}
\newcommand{\rp}{\rangle}
\newcommand{\cyl}{\mr{cyl}}

\newcommand{\sing}{\rm{sing}}
\newcommand{\ve}{\varepsilon}
\newcommand{\vp}{\varphi}
\newcommand{\KN}{\mathbin{\bigcirc\mspace{-15mu}\wedge\mspace{3mu}}}
\newcommand{\parexp}[1]{|#1|_{2, \exp}}
\newcommand{\polyd}[1]{\|#1\|_{2, \mr{mon}}}
\newcommand{\innerp}[2]{\left\lp #1,  #2\right\rp}

\DeclareMathOperator{\Rm}{Rm}
\DeclareMathOperator{\Rc}{Rc}

\begin{document}

\title{Singularity formation of complete Ricci flow solutions}

\author{Timothy Carson}
\address[Timothy Carson]{Google}
\email{timothycarson@google.com}

\author{James Isenberg}
\address[James Isenberg]{University of Oregon}
\email{isenberg@uoregon.edu}
\urladdr{http://www.uoregon.edu/$\sim$isenberg/}

\author{Dan Knopf}
\address[Dan Knopf]{University of Texas at Austin}
\email{danknopf@math.utexas.edu}
\urladdr{http://www.ma.utexas.edu/users/danknopf}

\author{Nata\v sa \v Se\v sum}
\address[Nata\v sa \v Se\v sum]{Rutgers University}
\email{natasas@math.rutgers.edu}
\urladdr{http://www.math.rutgers.edu/$\sim$natasas/}

\thanks{JI thanks the NSF for support in PHY-1707427.
DK thanks the Simons Foundation for support in Award 635293.
N\v S thanks the NSF for support in DMS-0905749 and DMS-1056387.}

\begin{abstract}
We study singularity formation of complete Ricci flow solutions, motivated by two applications: \textsc{(a)}
improving the understanding of the behavior of the \emph{essential blowup sequences} of
Enders--M\"uller--Topping~\cite{EMT} on noncompact manifolds, and \textsc{(b)} obtaining further
evidence in favor of the conjectured stability of generalized cylinders as Ricci flow singularity models.
\end{abstract}

\maketitle
\setcounter{tocdepth}{1}
\tableofcontents

\section{Introduction}		\label{Intro}

\subsection{Motivations}
Much is known about Ricci flow in dimensions $n=2,3$ and on compact manifolds. Much less is known about
solutions on higher-dimensional or noncompact manifolds. In this paper, using multiply-warped products,
we investigate various phenomena that occur in singularity formation on complete noncompact
solutions $\big(\mc M, g(t)\big)$ of Ricci flow, in arbitrary dimensions. We are most interested in singularities for which noncompactness
plays an essential role in the precise sense that the metric on any compact subset $K\subset\mc M$ remains nonsingular.
Our results for solutions of this type are found in Theorem~\ref{asymptotics}, Theorem~\ref{main}, and Corollary~\ref{shrink} below.

Our main applications of those results are found in Theorem~\ref{thm-spatial-infinity} and Corollary~\ref{necessary-and-sufficient}.
Briefly, we show that standard sequences of parabolic dilations at a singularity, which produce predictable subsequential limits
on compact solutions, as shown by Enders--M\"uller--Topping~\cite{EMT}, can yield unexpected limits for noncompact
solutions unless additional criteria are imposed. We make this statement precise below. In a second application,
Theorem~\ref{cor-stability0}, we prove a weak stability result for generalized cylinders evolving by Ricci flow, which is
motivated by well-known and much stronger results of Colding--Minicozzi~\cite{CM12, CM15} for mean curvature flow.

\subsection{Manifolds}
Let $(\mc B^n,g_{\mc B})$ be a complete noncompact Riemannian manifold.
For $\alpha\in\{1,\dots,A<\infty\}$,  let $(\mc F_\alpha^{n_\alpha},g_{\mc F_\alpha})$ be a collection of space forms,
and let $\mu_\alpha$ be constants such that $\mu_\alpha g_{\mc F_\alpha}=2\Rc[g_{\mc F_\alpha}]$.
Given functions $u_\alpha:\mc B^n\rightarrow\mb R_+$, there
is a warped product metric $g$ on the
manifold $\mc M^{\mc N}=\mc B^n\times\mc F_1^{n_1}\times\cdots\times\mc F_A^{n_A}$,
where $\mc N=n+{\sum_{\alpha=1}^A} n_\alpha$, given by
\begin{equation}	\label{eq:warped-product}
g=g_{\mc B} + \sum_{\alpha=1}^A u_\alpha g_{\mc F_\alpha}.
\end{equation}
For brevity, we omit the dimensions of the manifold $\mc M^{\mc N}$ and its factors $\mc F_\alpha^{n_\alpha}$ in what follows.

Under Ricci flow, the structure~\eqref{eq:warped-product} of the multiply-warped product metric is preserved,
and the base metric $g_{\mc B}$
and warping functions $u_\alpha$ evolve by the coupled diffusion-reaction system
\begin{subequations}	\label{eq:Ricci-flow-system}
\begin{align}
 \partial_t\,g_{\mc B} + 2 \Rc[g_{\mc B}] &= -2\sum_{\alpha=1}^A n_\alpha u_\alpha^{-1/2}\cv^2(u_\alpha^{1/2}),
 \label{eq:evolve-base}\\
 (\partial_t-\Delta)\,u_\alpha &=-\mu_\alpha-u_\alpha^{-1}|\cv u_\alpha|^2,\qquad\qquad (\alpha\in\{1,\dots,A\}).
 \label{eq:evolve-fiber}
\end{align}
\end{subequations}

\begin{remark}
Throughout this paper, undecorated geometric quantities are computed with respect to the metric $g$ on $\mc M$
and its Levi--Civita connection. In particular, the Laplacian in~\eqref{eq:Ricci-flow-system} denotes that of the metric $g$,
\emph{i.e.,} $\Delta\equiv\Delta_{\mc M}$, rather than the Laplacian $\Delta_{\mc B}$ of the metric $g_{\mc B}$ on the base.
Given any smooth function $\vp(x)$ depending only on $x\in\mc B$, the two differential operators are related by
\begin{equation}	\label{eq:compare-Laplacians}
\Delta_{\mc M}\vp = \Delta_{\mc B}\vp+\frac12 \sum_{\alpha=1}^A n_\alpha u_\alpha^{-1}\lp\cv u_\alpha,\cv\vp\rp,
\end{equation}
as follows easily from Claim \ref{claim:hessfmultiple} of Appendix~\ref{multiply-warped}.
\end{remark}

If some $u_\alpha(x,0)$ is a constant $a_\alpha$, then $u_\alpha(x,t)=a_\alpha-\mu_\alpha t$ is an
explicit solution of~\eqref{eq:evolve-fiber} for as long as the flow remains smooth. Since we are interested in studying 
perturbations of spatially homogeneous solutions, we set $a_{\alpha}=\inf_{x\in\mc B}u_\alpha(x,0)$ and define
$v_\alpha(\cdot,0):\mc B\rightarrow\mb R_+$ by
\begin{equation}	\label{eq:fix-v}
v_\alpha(x,0)=u_\alpha(x,0)-a_{\alpha},
\end{equation}
for $\alpha\in\{1,\dots,A\}$.  We observe that for as long as a smooth solution of
system~\eqref{eq:Ricci-flow-system} exists, the metric has the form
\begin{equation}	\label{eq:metric-preserved}
g(x,t)=g_{\mc B}(x,t) + \sum_{\alpha=1}^A \big\{(a_\alpha-\mu_\alpha t)+v_\alpha(x,t)\big\}g_{\mc F_\alpha}.
\end{equation}

\begin{remark}	\label{inf}
The construction outlined above ensures that $\inf_{x\in\mc B}v_\alpha(x,0)=0$.
Because our solutions are not compact, it is not automatic that $\inf_{x\in\mc B}v_\alpha(x,t)=0$ for $t>0$ for which a solution exists.
However, this follows from results we prove below.
\end{remark}

In Appendix~\ref{multiply-warped}, we compute the curvatures of $(\mc M,g)$.
Here, for $\alpha\in\{1,\dots,A\}$ and all $t\geq0$ that a Ricci flow solution exists, we define the functions
\begin{subequations}		\label{eq-def-quantity}
\begin{align}
\gamma_\alpha(x,t)&=|\cv v_\alpha(x,t)|^2,\\
\chi_\alpha(x,t)&=|\cv^2 v_\alpha(x,t)|^2_{g_{\mc B}},		\label{eq:define-chi} \\
\rho(x,t)&=\big|\Rm[g_{\mc B}](x,t)\big|^2_{g_{\mc B}},
\end{align}
\end{subequations}
where the first set of norms is computed with respect to the metric $g(\cdot,t)$ on the total space, but
the second and third sets are computed with respect to $g_{\mc B}$.
To motivate these quantities, we note that it follows from Remark~\ref{rem-curv} in Appendix~\ref{multiply-warped}
that there is a universal constant $C$ depending only on the dimensions such that
\begin{equation}	\label{eq:curv-near-cyl}
\left|\Rm[g]-\sum_{\alpha=1}^A u_\alpha^{-1}\Rm[g_{\mc F_\alpha}]\right|_g
	\leq C\left\{\rho^{1/2}+\sum_{\alpha=1}^A\left(u_\alpha^{-2}\gamma_\alpha+u_\alpha^{-1}\chi_\alpha^{1/2}\right)\right\}.
\end{equation}
So at points where the quantities $v_\alpha$ are small relative to $u_\alpha$, control of $\rho$,
$\gamma_\alpha/u_\alpha^2$, and $\chi_\alpha/u_\alpha^2$ indicates that the curvature is pointwise close
to that of an un-warped product.

\subsection{Main results}

In this paper, we assume that $\gamma_\alpha$, $\chi_\alpha$, and $\rho$
are bounded on our initial data in terms of  a constant $C_{\init}$ and functions
$G_\alpha$ and $H_\alpha$ in a manner that we call our Main Assumptions and  make
precise in  Section~\ref{sec:assumptions}. (Specifically, we use $G_\alpha$ to
bound $\gamma_\alpha$ and $H_\alpha$ to bound $\chi_\alpha$.)
\smallskip

Our first result provides an asymptotic description
of all solutions of Ricci flow originating from initial data that satisfy those assumptions.
Specifically, it shows in a precise sense that the asymptotics of the original data are preserved:

\begin{theorem}	\label{asymptotics}
Let $\big(\mc M,g(t)\big)$ be a solution of the Ricci flow system~\eqref{eq:Ricci-flow-system} that
originates from initial data satisfying our Main Assumptions and exists for $t\in[0,T_{\mr{small}}]$.

There exists a constant $C_*=C_*(n,n_\alpha,C_\init)$ such that for $t\in[0,\min\{T_{\mr{small}},C_*^{-1}\})$,
the metric can be written as
\begin{align*}
g(x,t)&= \big(1+\mc O(1)\big) g_{\mc B}(x,0)\\
	&\qquad+\sum_{\alpha=1}^A
	\left\{ (a_\alpha-\mu_\alpha t) +
	\left(1+\mc O\left(\frac{G_\alpha\left(v_\alpha(x,0)\right)}{v_\alpha^2(x,0)}\right)\right)
	v_\alpha(x,0)\right\} g_{\mc F_\alpha}.
\end{align*}
\end{theorem}
We note that the Main Assumptions imply that the terms $G_\alpha\big(v_\alpha(x,0)\big)/v_\alpha^2(x,0)$
are bounded. By those assumptions, those terms bound $|\cv\log v_\alpha(x,0)|^2$, which in turn implies
that the functions $v_\alpha(\cdot,0)$ can decay at most exponentially (see Remark~\ref{positive} below).
In fact, if the functions $G_\alpha$ are chosen so that the quantities $G_\alpha\big(v_\alpha(x,0)\big)/v_\alpha^2(x,0)$
are comparable to $|\cv\log v_\alpha(x,0)|^2$, then $G_\alpha(x,0)/v_\alpha^2(x,0)\searrow0$ as $v_\alpha(x,0)\searrow0$ if
and only if $v_\alpha(\cdot,0)$ decays more slowly than exponentially.
\smallskip

We prove Theorem~\ref{asymptotics} in the course of proving the following stronger but more technical result:

\begin{theorem}	\label{main}
Let $(\mc M,g_\init)$ satisfy the Main Assumptions stated in Section~\ref{sec:assumptions}.
Then there exists a constant $C_*=C_*(n,n_\alpha,C_\init)$ such that the following are true:

A solution
\[
g(x,t)=g_{\mc B}(x,t) + \textstyle\sum_{\alpha=1}^A \big\{a_\alpha-\mu_\alpha t+v_\alpha(x,t)\big\}g_{\mc F_\alpha}
\]
of the Ricci flow initial value problem with $g(x,0)=g_\init(x)$ exists with curvatures bounded in space
at all times $t\in[0, T_*)$, where $T_* := \min\{T_{\sing},C_*^{-1}\}$, and $T_{\sing}$ is the (finite) singularity time,
\emph{i.e.,} the maximal existence time of a smooth solution.

The $v_\alpha$ are uniformly equivalent for $t\in[0, T_*)$.
Specifically, one has
\[
\frac{1}{C_*}v_\alpha(x,t)\leq v_\alpha(x,0)\leq C_* v_\alpha(x,t).
\]

Moreover, for each $x\in\mc B$ and $t\in[0, T_*)$, one has
\begin{subequations}	\label{eq:main-estimates}
\begin{equation}
\rho(x,t)\leq C_\init\,(1+C_*t),
\end{equation}
and for $\alpha\in\{1,\dots,A\}$,
\begin{align}
\gamma_\alpha(x,t)&\leq C_{\init}\, \left(1 + C_* t \,\frac{G_\alpha\big(v_\alpha(x,t)\big)}{v_\alpha^2(x,t)}\right)\,
	G_{\alpha}(v_{\alpha}(x,t)),\\
\chi_\alpha(x,t)&\leq C_{\init}\,(1+C_*t)\,H_\alpha\big(v_\alpha(x,t)\big),
\end{align}
\end{subequations}
where $G_\alpha$ and $H_\alpha$ are functions specified in the Main Assumptions.
\end{theorem}
We prove Theorem~\ref{main} in Section~\ref{sec:main-proof} below after precisely stating our assumptions
in Section~\ref{sec:assumptions} and establishing preliminary estimates in
Sections~~\ref{basic}--\ref{main_estimates}.

If $C_*^{-1}<T_{\sing}$, then the theorem cannot describe the solution up to the singular time. However, we can
always arrange that it does apply up to $T_{\sing}$, as we now explain.
A key strength of the theorem is that the constant $C_*$ is independent of the quantities $a_\alpha$.
One sees from~\eqref{eq:curv-near-cyl} that the curvature can be very large if some $a_\alpha$ is very small. But even in that
case, the bounds~\eqref{eq:main-estimates} persist. This leads directly to our next result.  We 
let $\varsigma$ be such that $a_\varsigma/\mu_\varsigma = \min\{a_\alpha/\mu_\alpha\colon \mu_\alpha>0\}$.
By~\eqref{eq:metric-preserved}, the metric on $\mc F_\varsigma$ has the form
$\big\{(a_\varsigma-\mu_\varsigma t)+v_\varsigma(x,t)\big\}g_{\mc F_\varsigma}$. By Remark~\ref{inf},
$\inf v_\varsigma(\cdot,0)=0$, and by Theorem~\ref{main}, this infimum is preserved. Thus the solution
cannot exist past the formal singularity time $T_{\form} := a_\varsigma/\mu_\varsigma$.
Hence we have the following Corollary.

\begin{corollary}	\label{shrink}
There exist initial data $(\mc M,g'_\init)$ satisfying the Main Assumptions stated in Section~\ref{sec:assumptions}
with the same constant $C_{\init}$, the same initial values $v_\alpha$, the same real-valued functions $G_\alpha$
and $H_\alpha$, but with changed constants $a_\alpha$,  such that the conclusions
of Theorem~\ref{main} hold for the Ricci flow evolution of $(\mc M,g'_\init)$ at all times $[0,T_{\sing})$.
Moreover, $T_{\sing} = T_{\form}$; there are no finite singular points in space; and the singularity is
Type-I and occurs at spatial infinity.
\end{corollary}
A proof of this Corollary is found in Section~\ref{sec:main-proof}, following the proof of Theorem~\ref{main}.

\smallskip

A schematic outline of our proof of Theorems~\ref{main} is as follows.
The proof relies on two pairs of supporting results, with
Propositions \ref{lem:Tim-15} and \ref{cor-improvement} composing the first pair, and
Propositions~\ref{lem:Tim-16} and \ref{uniform} composing the second. In the process, we
obtain Theorem~\ref{asymptotics} as a consequence of the arguments we employ to prove Proposition~\ref{lem:Tim-16}.

Standard short-time existence results give us a smooth Ricci flow
solution on some time interval $[0,T_{\min}]$, with some curvature bound. Propositions~\ref{lem:Tim-15} and \ref{cor-improvement} take as their input a 
curvature bound on 
$[0,T_{\min}]$; they output\  linear growth estimates for $\rho,\gamma_\alpha,\chi_\alpha$ on an interval $[0,T_1]\subseteq[0,T_{\min}]$,
albeit with a possibly large constant that depends on the input curvature bound.
As noted below the statement of Theorem~\ref{main}, we ultimately do not want estimates that directly depend on the curvature.
The fact that we get linear growth estimates for $\rho,\gamma_\alpha,\chi_\alpha$, however, lets us then apply
Propositions~\ref{lem:Tim-16} and \ref{uniform},
which take as their
input uniform bounds on a suitable subinterval $[0,T_2]\subseteq[0,T_1]$ and yield the conclusions of the theorem
on some time interval $[0,T_3]\subseteq[0,T_2]$. Finally, we use an ``open-closed'' argument to show
that the supremum of $t>0$ such that the Theorem holds cannot be too small,  \emph{i.e.,} that it extends to
$\min\{T_{\sing},C_*^{-1}\}$.

\subsection{Applications}
\subsubsection{Essential blowup sequences on noncompact manifolds}
The main application of Theorem~\ref{main} that we have in mind in this paper
is to obtain new insights into blowup limits of singularities on complete noncompact manifolds. 
We rigorously explore the phenomena that occur if finite-time singularities form at spatial
infinity on noncompact manifolds. More precisely, we construct complete Ricci flow solutions for which Type-I singularities
occur at spatial infinity and which do not have any Type-I singular points.
The existence of such (singly-warped) examples has been conjectured in \cite{EMT}.
We show that for each of our (doubly-warped) examples, taking a blowup limit along some essential blow up sequence
(see Section~\ref{sec-application} for precise definitions) yields a gradient shrinking Ricci soliton in the
subsequential limit, whereas taking a subsequential limit along some other essential blow up sequence
yields a complete ancient solution that is not a soliton. We summarize these results in the following theorem: 

\begin{theorem}
\label{thm-spatial-infinity}
There exist complete, noncompact, $\kappa$-noncollapsed Ricci flow solutions $(\mc M, g(t))$, with
$\mc M := \mathbb{R}\times \mc S^p\times \mc S^p$, that develop Type-I singularities at spatial infinity.

On each of these solutions, there exist essential blowup sequences along which a blowup limit yields a nontrivial
gradient shrinking Ricci soliton, and there exist essential blowup sequences along which no blowup limit can be 
a gradient shrinking Ricci soliton.
\end{theorem}

The key idea is that on noncompact Ricci flow solutions, there can be essential blowup sequences with no
Type-I singular point limit, 
 and these sequences may or may not have nontrivial gradient Ricci soliton limits.
However, one can obtain soliton limits by imposing another condition. Indeed, we show the following in the proof
of Theorem~\ref{thm-spatial-infinity}:

\begin{corollary}	\label{necessary-and-sufficient}
Under the conditions of Theorem~\ref{thm-spatial-infinity}, a blowup limit of the flow along a sequence
$(x_j,t_j)$ with $|x_j|\rightarrow\infty$ and $t_j\rightarrow a_*$ is a nontrivial gradient soliton if and only if
\[
\lim_{j\to\infty}\frac{|\Rm(x_j,t_j)|}{\sup_{\mc M} |\Rm(\cdot,t_j)|} = 1.
\]
\end{corollary}
In other words, to obtain a nontrivial gradient shrinking soliton limit, it is both necessary and sufficient that
$|\Rm(x_j,t_j)|\to\sup_{\mc M}|\Rm(\cdot,t_j)|$ as $|x_j|\to\infty$. Clearly, the subsequences we construct
in Theorem~\ref{thm-spatial-infinity} that fail to have soliton limits do not satisfy this condition.
\smallskip

We obtain a related result for solutions on $\mc M = \mathbb{R}\times \mc S^1\times \mc S^p$. These are
not $\kappa$-noncollapsed, hence do not have blowup limits except as \'etale groupoids, in the
sense considered by Lott~\cite{Lott10}.
\smallskip

We believe that the arguments we use to prove Theorem~\ref{thm-spatial-infinity} could 
easily be extended to construct $\kappa$-noncollapsed examples on $\mb R^k\times\mc S^p\times\mc S^p$ for any $p\geq2$
with the same properties that (a) their singularities occur at spatial infinity and that (b) distinct subsequential blowup limits are
possible.

\subsubsection{Weak stability of generalized cylinders}

Stability of cylinders $\mb R^k\times \mc S^p$ under Ricci flow is a subtle question.
Even though a round cylinder $\mb R^k\times\mc S^p$ is expected to be a stable singularity model in a suitable sense,
it is not immediately clear how to define its stability.
In the case of mean curvature flow, it is shown in~\cite{CM12} that the only entropy-stable\footnote{See definitions~(0.5) and~(0.6) in~\cite{CM12}.}
shrinkers are spheres, hyperplanes, and generalized cylinders.
Currently, there is no analogue of such a result in Ricci flow. Accordingly, we adopt the following:

\begin{definition}
We say a solution $g(\cdot,t)$ of Ricci flow  is \emph{weakly stable} if for every $\epsilon > 0$, there exists $\delta > 0$
such that for any other Ricci flow solution $\tilde{g}(\cdot,t)$ satisfying $\|g(\cdot,0) - \tilde{g}(\cdot,0)\|_{C^0} < \delta$,
one has $\|g(\cdot,t) - \tilde{g}(\cdot,t)\|_{C^0} < \epsilon$ for all $t \ge 0$ that both solutions exist.                                                
\end{definition}

We prove the following result, which is stated more precisely as Theorem~\ref{cor-stability} in the text below.

\begin{corollary}	\label{cor-stability0}
Ricci flow of a direct product metric $g_{\cyl}$ on $\mb R^k\times \mc S^p$ is weakly stable with respect to
admissible perturbations of $g_{\cyl}.$\footnote{These are understood in the sense of Definition~\ref{def-adm-pert}, below.}

Moreover, if $g(\cdot,0)$ is an admissible perturbation of $g_{\cyl}(\cdot,0)$, then both flows
$g(\cdot,t)$ and $g_{\cyl}(\cdot,t)$ develop a singularity at the same finite time and that they stay close to each
other in the $C^0$ norm up to that singular time.
\end{corollary} 

\begin{remark}
We note that the proof and conclusion of Corollary~\ref{cor-stability0} also apply for any direct product metric on 
$\mb R^k\times \mc S^p\times \mc S^q$, for any nonnegative integers $p$ and $q$, as long as at least one of them is nonzero. 
\end{remark}

\begin{remark}	\label{positive}
We further note that part~\eqref{eq:grad-bound} of the Main Assumptions detailed in
Section~\ref{sec:assumptions}, requires $|\cv\log v_{\alpha,\init}|^2$ to be bounded, which
implies that $\inf u_{\alpha,\init}$ cannot be attained. This means that Theorem~\ref{main} is primarily useful in analyzing
singularities that occur at spatial infinity.

Given initial data in which $\inf u_\alpha$ is attained in a compact set, one could adjust $a_{\alpha,\init}$ downward
in order to apply the Theorem. However, its output would not be sharp in that case, because it then cannot describe the 
developing singularity all the way up to the singular time.
\end{remark}

\section{Assumptions and preliminary estimates}

\subsection{Assumptions}	\label{sec:assumptions}
We begin by establishing some notation.

Given a smooth function $\vp : \mb R_+ \to \mb R_+$, we define
\[
\polyd{\vp} = \sup_{s\in\mb R_+} \left(1 +  \frac{s |\vp'(s)|}{\vp(s)} +  \frac{s^2 |\vp''(s)|}{\vp(s)}\right).
\]
We caution the reader that this is not a norm. The double bars are a reminder that $\polyd{\cdot}$ is a supremum
rather than a pointwise bound. The subscript is a reminder that $\polyd{\vp}$ depends on two derivatives
of $\vp$, and that the quantity in parenthesis is constant if $\vp$ is a monomial.

Given a smooth function $\psi : \mc B \times [0, T] \to \mb R_+$, we define
\[
\parexp{\psi} = \frac{\Big|\heat\psi\Big|}{\psi}+\frac{|\cv \psi|^2}{\psi^2}.
\]
The single bars in $\parexp{\cdot}$ are a reminder that it is a pointwise bound, \emph{i.e.,} a function of
$x\in\mc B$ rather than a supremum. The subscript is a reminder that
$\parexp{\cdot}$ depends on two derivatives, and that $|\cv \psi|^2/\psi^2$ is constant in space if
$\psi(x,t)=e^{d_{g(t)}(x',x)}$ for some $x'\in\mc B$, where $d_{g(t)}(x',x)$ represents distance
with respect to the metric $g(t)$.

In Section~\ref{basic}, we state some useful properties satisfied by $\polyd{\cdot}$ and $\parexp{\cdot}$.
\smallskip

We next define
\begin{equation}	\label{eq-set-G}
\mc G=\left\{G:\mb R_+\rightarrow\mb R_+\colon \| G\|_{\mc G}:= \polyd{G}+
\sup_{s\in\mb R_+}\frac{G(s)}{s^2}<\infty\right\}.
\end{equation}
We note that $s^2\in\mc G$, so $\mc G\neq\emptyset$.  We again caution the reader that we are once more
using nonstandard notation: the symbol $\|\cdot\|_{\mc G}$  defined here is not a norm, and $\mc G$ is not a
vector space.

Any choices of $G_\alpha\in\mc G$ generate associated functions $H_\alpha\in\mc G$ defined by
\begin{equation}	\label{eq-H}
H_\alpha[s_1,\dots, s_A](s_\alpha)=\left(\sum_{\beta=1}^A\frac{G_\beta(s_\beta)}{s_\beta^2}\right)
G_\alpha(s_\alpha).
\end{equation}
The notation reflects the fact that $H_\alpha$ is intended to control the geometry on the fiber $\mc F_\alpha$,
but inputs information from the functions $G_1,\dots,G_A$ used to control the geometry of all fibers $\mc F_1,\dots,\mc F_A$.
For brevity, we write $H_\alpha[s_1,\dots , s_A](s_\alpha)\equiv H_\alpha(s_\alpha)$ below.
The mnemonic theme is that we find it convenient to use $G_\alpha,H_\alpha\in\mc G$ to
control gradient and Hessian terms in~\eqref{eq:grad-bound} and~\eqref{eq:Hessian-bound}, respectively. 
We assume below that our choices of $G_\alpha$ satisfy the inequalities
$\|G_{\alpha}\|_{\mc G} \le \bar{C}_{\alpha}$ for some constants $\bar C_\alpha$,
$\alpha\in \{1,\dots,A\}$.
\smallskip
 
 Throughout this paper, we assume that our initial data consist of a metric
\[
g_{\init}(x)=g_{\mc B}(x,0) + \textstyle\sum_{\alpha=1}^A \big\{a_\alpha+v_\alpha(x,0)\big\}g_{\mc F_\alpha}
\]
on the manifold $\mc M=\mc B\times\mc F_1\times\cdots\times\mc F_A$ satisfying the following:
\begin{assumption}
There exist a constant $C_\init$ and functions $G_\alpha\in\mc G$ such that for $\alpha\in\{1,\dots,A\}$,
\begin{subequations}
\begin{align}
\| G_\alpha\|_{\mc G}&\leq C_\init,\\
\gamma_\alpha(x,0) &\leq C_{\init}\,G_\alpha\big(v_\alpha(x,0)\big)&\mbox{ for all }x\in\mc B,	\label{eq:grad-bound}\\
\chi_\alpha(x,0) &\leq C_{\init}\,H_\alpha\big(v_\alpha(x,0)\big)&\mbox{ for all }x\in\mc B,	\label{eq:Hessian-bound}\\
\rho(x,0)&\leq C_\init&\mbox{ for all }x\in\mc B.
\end{align}
\end{subequations}
We further assume that $|\nabla\Rm[g(\cdot,0)]|_{g(\cdot,0)}$ is bounded and
that at least one $\mu_\alpha>0$, 
\emph{i.e.,} that at least one fiber is a space form of positive Ricci curvature.
\end{assumption}
We note that our choices of $G_\alpha\in\mc G$ may depend on the initial data, and that it follows from
our main results that the choice $\mu_\alpha>0$ forces a singularity at a time $T_{\sing}<\infty$.

\subsection{Basic inequalities}	\label{basic}

It is not difficult to verify the following useful properties of $\polyd{\cdot}$ and $\parexp{\cdot}$:
\begin{equation}		\label{eq:easycalc}
\begin{aligned}
\polyd{\vp_1+\vp_2} &\leq \polyd{\vp_1} + \polyd{\vp_2},\\
\polyd{\vp_1\vp_2} &\leq \polyd{\vp_1}\polyd{\vp_2},\\
\polyd{\vp_1 \circ \vp_2} &\leq \polyd{\vp_1}\polyd{\vp_2}^2,\\
\parexp{\vp \circ \psi} &\leq \polyd{\vp}^2 \parexp{\psi},   \\
\parexp{\psi_1 \psi_2} &\leq \parexp{\psi_1} + \parexp{\psi_2}, \\          
\parexp{\psi_1 +\psi_2} &\leq 2\, (\parexp{\psi_1} + \parexp{\psi_2}).
\end{aligned}
\end{equation}
We explicitly verify the fourth inequality, whose proof is slightly less straightforward than the
proofs of the others.

\begin{proof}
Let $\psi:\mc M\times[0,T]\rightarrow\mb R_+$ and $\vp:\mb R_+\rightarrow\mb R_+$.
Then $(\vp\circ\psi)_t=\vp'(\psi)\psi_t$, $\cv_i(\vp\circ\psi)=\vp'(v)\cv_i\psi$, and
$\Delta(\vp\circ\psi)=\vp'(\psi)\Delta\psi + \vp''(\psi)|\cv\psi|^2$.
Thus one has
\begin{align*}
\frac{\heat(\vp\circ\psi)}{\vp\circ\psi}+\frac{|\cv(\vp\circ\psi)|^2}{(\vp\circ\psi)^2}
	&=\frac{\vp'(\psi)\heat\psi-\vp''(\psi)|\cv\psi|^2}{\vp(\psi)}+\frac{(\vp'(\psi))^2|\cv\psi|^2}{(\vp(\psi))^2}\\
	&=\frac{\psi \vp'(\psi)}{\vp(\psi)}\frac{\heat \psi}{\psi}-\frac{\psi^2\vp''(\psi)}{\vp(\psi)}\frac{|\cv\psi|^2}{\psi^2}\\
	&\quad+\Big\{\frac{\psi\vp'(\psi)}{\vp(\psi)}\Big\}^2\frac{|\cv\psi|^2}{\psi^2},
\end{align*}
from which it is easy to see that $\parexp{\vp \circ \psi} \leq \polyd{\vp}^2 \parexp{\psi}$.
\end{proof}

\subsection{Differential inequalities}

We now estimate the evolution equations of the quantities we work with throughout this paper:
$\gamma_{\alpha}$, $\chi_{\alpha}$, and $\rho$.

\begin{lemma}
\label{lemma-ev-eq}
If $\gamma_{\alpha}$, $\chi_{\alpha}$, and $\rho$ are as in \eqref{eq-def-quantity}, then there exists a uniform constant $C_N$  that depends only on the dimension vector $\vec N=(n,n_\alpha)$ such that we have the estimates
\begin{equation}
\label{eq-nabla-v}
\heat \gamma_{\alpha} \le -\frac 12 \, \frac{|\nabla \gamma_{\alpha}|^2}{\gamma_{\alpha}} + 6\, \left(\frac{\gamma_{\alpha}}{u_{\alpha}^2}\right)\, \gamma_{\alpha} ,
\end{equation}
\begin{equation}	\label{kappa-evolution}
\heat\chi_{\alpha} \le -\frac12\, \frac{|\nabla\chi_{\alpha}|^2}{\chi_{\alpha}} + C_N L \chi_{\alpha} + C_N L \sum_{\beta=1}^A \frac{\gamma_{\beta}}{u_{\beta}^2}\, \gamma_{\alpha},
\end{equation}
and
\begin{equation}	 
\label{eq-rho}
\heat\rho \le -\frac{|\nabla\rho|^2}{\rho} + C_N L^3,
\end{equation}
where
\[
L := \rho^{1/2} + \sum_{\beta=1}^A \frac{\gamma_{\beta}}{u_{\beta}^2} + \sum_{\beta =1}^A \frac{\chi_{\beta}^{1/2}}{u_{\beta}}.
\]
\end{lemma}
The helpful structure here is that we have negative gradient terms in all three equations.
In~\eqref{eq-nabla-v} and~\eqref{kappa-evolution}, we also have what may be regarded as linear terms
with coefficients that can be bounded in terms of the quantities under consideration;
in~\eqref{kappa-evolution} and~\eqref{eq-rho}, we have inhomogeneous terms that may be similarly bounded.

\begin{proof}
In the proof, we use the same symbol $C_N$ to denote constants that might differ from
line to line but that all depend only on the dimension vector $\vec N=(n,n_\alpha)$. 

An easy computation (see Appendix~\ref{evolve-curvatures}) implies that
\[
\heat \gamma_{\alpha} =  - 2|\nabla^2 v_{\alpha}|^2 
+ 2\,\frac{|\nabla v_{\alpha}|^4}{u_{\alpha}^2} - 4\,\frac{\nabla^2v_{\alpha}(\nabla v_{\alpha}, \nabla v_{\alpha})}{u_{\alpha}}.
\]
Using Cauchy--Schwarz and Kato's inequality ($|\nabla |\nabla v_{\alpha}|| \le |\nabla^2 v_{\alpha}|$), we get \eqref{eq-nabla-v}.
\smallskip

To obtain~\eqref{kappa-evolution}, note that in Appendix \ref{evolve-curvatures} we compute that
\begin{equation*}	
\begin{split}
\heat\chi_{\alpha}&\leq-2|\cv^3 v_{\alpha}|_{g_{\mc B}}^2+4\Rm_{\mc B}(\cv^2 v_{\alpha}, \cv^2 v_{\alpha})+2u_{\alpha}^{-2}\gamma_{\alpha}\chi_{\alpha}\\
&\quad-2u_{\alpha}^{-3}\lp\cv v_{\alpha},\cv\gamma_{\alpha}\rp\gamma_{\alpha}+4u_{\alpha}^{-2}\lp\cv^2 v_{\alpha},\cv v_{\alpha}\otimes\cv\gamma_{\alpha}\rp\\
&\quad-2u_{\alpha}^{-1}\lp\cv^2 v_{\alpha},\cv^2 \gamma_{\alpha}\rp_{g_{\mc B}}
+Nu_{\alpha}^{-2}\gamma_{\alpha}\big\{-\chi_{\alpha}+\frac14u_{\alpha}^{-1}\lp\cv\ v_{\alpha},\cv\gamma_{\alpha}\rp\big\}\\
&\quad+C_N \Big(\sum_{\alpha=1}^A |\cv \log u_\alpha|^2\Big) |\cv^2v_{\alpha}| |\cv^2 v_{\alpha}|_{g_{\mc B}}.
\end{split}
\end{equation*}
An easy computation using results about the Levi-Civita connection $\Gamma$
derived in Appendix~\ref{multiply-warped} yields
\[
  -|\nabla^3 v_{\alpha}|^2_{g_{\mc B}}
  \le
  -\frac
  {|\nabla\, |\nabla^2 v_{\alpha}|^2_{g_{\mc B}}\,|^2_{g_{\mc B}}}
  {2\chi_{\alpha}}
  = -\frac{|\nabla\chi_{\alpha}|^2_{g_{\mc B}}}{2\chi_{\alpha}} = -\frac{|\nabla\chi_{\alpha}|^2}{2\chi_{\alpha}}\]
and
\begin{equation*}
\begin{split}
\left|u_{\alpha}^{-1}\lp\cv^2 v_{\alpha},\cv^2 \gamma_{\alpha}\rp_{g_{\mc B}}\right| &\le C_N \, \left(\frac{\chi_{\alpha}^{3/2}}{u_{\alpha}} +  \frac{\gamma_{\alpha}^{1/2}\, \chi_{\alpha}^{1/2}}{u_{\alpha}}\, |\nabla^3 v_{\alpha}|_{g_{\mc B}}\right) \\
&\le  C_N \, \left(\frac{\chi_{\alpha}^{3/2}}{u_{\alpha}} +  \frac{\gamma_{\alpha}\, \chi_{\alpha}}{u_{\alpha}^2}\right) + |\nabla^3 v_{\alpha}|^2_{g_{\mc B}}.
\end{split}
\end{equation*}
Again using results about the Hessian from Appendix \ref{multiply-warped}, we see that
\[|\nabla^2 v_{\alpha}| \le C_N\, \left(\chi_{\alpha}^{1/2} + \frac{\gamma_{\beta}^{1/2}}{u_{\beta}}\, \gamma_{\alpha}^{1/2}\right),\]
implying that
\[|\nabla^2 v_{\alpha}| |\nabla^2 v_{\alpha}|_{g_{\mc B}} \le C_N \, \left(\chi_{\alpha} + \sum_{\beta=1}^A \frac{\gamma_{\beta}}{u_{\beta}^2}\, \gamma_{\alpha}\right),\]
where we use the Cauchy--Schwarz inequality.
Putting these estimates together yields
\begin{equation*}
\begin{split}
\heat \chi_{\alpha} &\le - \frac{|\nabla \chi_{\alpha}|^2}{2\chi_{\alpha}} + C_N \chi_{\alpha}\, \left(\rho^{1/2} + \frac{\gamma_{\alpha}}{u_{\alpha}^2} + \frac{\chi_{\alpha}^{1/2}}{u_{\alpha}} + \sum_{\beta=1}^A \frac{\gamma_{\beta}}{u_{\beta}^2}\right) \\
&+ C_N \, \left(\frac{\chi_{\alpha}^{1/2} \gamma_{\alpha}^2}{u_{\alpha}^3} +  \Big(\sum_{\beta=1}^A \frac{\gamma_{\beta}}{u_{\beta}^2}\Big)^2 \gamma_{\alpha}\right) \\
&\le  - \frac{|\nabla \chi_{\alpha}|^2}{2\chi_{\alpha}}  + C_N \, L\,\chi_{\alpha} + C_N\, \gamma_{\alpha}\, \sum_{\beta=1}^A \frac{\gamma_{\beta}}{u_{\beta}^2}\, \left(\sum_{\beta=1}^A \frac{\chi_{\beta}^{1/2}}{u_{\beta}} + \sum_{\beta=1}^A \frac{\gamma_{\beta}}{u_{\beta}^2}\right) \\
&\le  - \frac{|\nabla \chi_{\alpha}|^2}{2\chi_{\alpha}}  + C_N \, L\,\chi_{\alpha} + C_N \, L\,\sum_{\beta=1}^A \frac{\gamma_{\beta}}{u_{\beta}^2}\, \gamma_{\alpha},
\end{split}
\end{equation*}
as claimed.
\smallskip

Finally, as in Appendix \ref{evolve-curvatures}, denote by $\mc H$ the (integrable) horizontal distribution of  $\mc M$ and by
$\Rm_{\mc H\otimes\mc H}$ the restriction
\[
\Rm_{\mc H\otimes\mc H}:=\Rm \big|_{\mc H \otimes T \mc M \otimes T \mc M \otimes \mc H}.
\]
Our computation in Appendix \ref{evolve-curvatures} shows that $\rho$ evolves by
\begin{equation*}	
\begin{split}
\heat\rho &\leq-2|\cv\Rm|^2_{g_{\mc B}}+C_n\rho^{3/2}\\
&\quad+2\sum_{\alpha=1}^A n_\alpha\Big\{
u_\alpha^{-2}\Rm_{\mc B}(\cv^2v_\alpha,\cv^2v_\alpha)\\
&\qquad\qquad\qquad-2u_\alpha^{-3}\Rm_{\mc B}(\cv^2v_\alpha,\cv v_\alpha\otimes\cv v_\alpha)\Big\}\\
&\quad+C_N \Big(\sum_{\alpha=1}^A |\cv \log u_\alpha|^2\Big) |\Rm|_{g_{\mc B}} |\Rm_{\mc H\otimes\mc H}|_g.
\end{split}
\end{equation*}
Claim~\ref{lem:Tim} in Appendix~\ref{evolve-curvatures} shows that $\cv\Rm$ vanishes if exactly one index is vertical.
Thus by Kato's inequality for tensors, we have 
\[-|\nabla\Rm|_{g_{\mc B}}^2 \le - |\nabla |\Rm|_{g_{\mc B}}|^2_{g_{\mc B}} = -\frac{|\nabla\rho|^2_{g_{\mc B}}}{\rho} = -\frac{|\nabla\rho|^2}{\rho}.\]
Moreover, using our computations of curvature components in Appendix \ref{multiply-warped}, we immediately get
\[|\Rm_{\mc H\otimes\mc H}|_g \le C_N\, \left(\sum_{\beta=1}^A \frac{\gamma_{\beta}}{u_{\beta}^2}
+ \sum_{\beta=1}^A \frac{\chi_{\beta}^{1/2}}{u_{\beta}} + \rho^{1/2} \right).\]
All of these together imply that
\begin{equation*}
\begin{split}
\heat\rho &\le -\frac{|\nabla\rho|^2}{\rho} + C_N \rho^{3/2} + C_N\sum_{\beta=1}^A \frac{\rho^{1/2}\chi_{\beta}}{u_{\beta}^2} + C_N \sum_{\beta=1}^A \frac{\rho^{1/2}\chi_{\beta}^{1/2}\gamma_{\beta}}{u_{\beta}^3} \\
&+ C_N \sum_{\beta=1}^A \frac{\gamma_{\beta}}{u_{\beta}^2} \rho^{1/2} \left(\sum_{\beta=1}^A \frac{\gamma_{\beta}}{u_{\beta}^2}
+ \sum_{\beta=1}^A \frac{\chi_{\beta}^{1/2}}{u_{\beta}} + \rho^{1/2}\right)\\
&\le  -\frac{|\nabla\rho|^2}{\rho} + C_N L^3,
\end{split}
\end{equation*}
yielding \eqref{eq-rho}. This completes the proof.
\end{proof}

\section{Analysis}	\label{sec:analysis}
In this section, we prove estimates for solutions of parabolic equations on noncompact manifolds evolving by
Ricci flow. Among the results we obtain below are Propositions~\ref{lem:Tim-15}, \ref{cor-improvement}, \ref{lem:Tim-16},
and~\ref{uniform} which, as discussed in the introduction, play a major role in the proof of  Theorem~\ref{main}.

\subsection{A noncompact maximum principle}		\label{maximum_principle}

The goal of our first result, Lemma~\ref{lem:Tim-17},  is to obtain estimates for a function $U$ in terms of
a ``comparison function'' $V$
and a ``control function'' $W$. For example, we often take $U$ to be a function that we want to estimate on
a short time interval $[t_0,t_1]$, $V$ to be the same function at the initial  time $t_0$, and $W$ to be a large constant
that depends on bounds for the curvatures on $[t_0,t_1]$. Our proof of the lemma proceeds by applying a
noncompact maximum principle to the quantity $U/V$, thereby allowing us to bound it suitably from above.
We use Lemma~\ref{lem:Tim-17} extensively in the proofs below.

\begin{lemma}	\label{lem:Tim-17}
Let $\big(\mc M,g(t)\big)$ be a smooth solution of Ricci flow for $t\in[0,T]$, and 
let $U,V,W:\mc M \times [0, T] \rightarrow\mb R_+$ be smooth functions.
Suppose that there exist constants $0<c<1<C$ such that
\begin{align*}
\heat U&\leq C(UW+VW)-c\frac{|\cv U|^2}{U},\\
\frac{\big|\heat V\big|}{V}+\frac{|\cv V|^2}{V^2}&\leq CW,\\
\big|\heat W\big|+|\cv W|^2&\leq CW, \\
W\leq C,
\end{align*}
where the Laplacian and norms above are computed with respect to the solution $g(t)$ of Ricci flow.

Then there exist $\lambda=\lambda(c,C)$ and $T'=T'(c,C,T)\in(0,T]$ such that for all $t\in[0,T']$,
\[
\heat\left\{\frac{U}{V} -\lambda t\left(1 + \frac{U}{V}\right)W\right\}\leq0.
\] 
\smallskip

Moreover, if there exist a point $x'\in\mc B$ and a constant $C'$ such that one has
$U(x,t)\leq C'e^{C' d^2_{g(t)}(x',x)}V(x, t)$ on $[0,T']$, then for $t\in[0,T']$,
\[
\frac{U(x,t)}{V(x,t)}
\leq 
\sup_{y\in\mc M}\frac{U(y,0)}{V(y,0)}
+ 2\lambda t  W(x,t)\, \left(1 + \sup_{y\in \mc M} \frac{U(y,0)}{V(y,0)}\right).
\]
\end{lemma}

\begin{proof}
We define $X=U/V$ and compute that
\[
\heat X = \frac{\heat U}{V}+2X\frac{\lp\cv U, \cv V\rp}{UV}-2X\frac{|\cv V|^2}{V^2}-X\frac{\heat V}{V}.
\]
We split the second term on the \textsc{rhs} above as follows:
\begin{equation}	\label{eq:split}
2X\frac{\lp\cv U, \cv V\rp}{UV}=(2-c) X\frac{\lp\cv U, \cv V\rp}{UV}+c X\frac{\lp\cv U, \cv V\rp}{UV}.
\end{equation}
In what follows, we denote by $C'=C'(c,C)$ a constant that may change from line to line.
We use the weighted Cauchy--Schwarz inequality to estimate the first term on the \textsc{rhs} of~\eqref{eq:split} by
\[
(2-c)X\frac{|\lp\cv U, \cv V\rp|}{UV}\leq \frac{c}{2}\frac{|\cv U|^2}{UV}+C'\frac{|\cv V|^2}{V^2}X,
\]
and rewrite the second term as
\[
cX\frac{\lp\cv U, \cv V\rp}{UV}=\frac{c}{2}\Big(\frac{|\cv U|^2}{UV}+X\frac{|\cv V|^2}{V^2}-\frac{|\cv X|^2}{X}\Big),
\]
obtaining
\[
\heat X \leq \frac{1}{V}\left\{\heat U+c\frac{|\cv U|^2}{U}\right\}
+X\left\{\frac{\Big|\heat V\big|}{V}+C'\frac{|\cv V|^2}{V^2}\right\}-\frac{c}{2}\frac{|\cv X|^2}{X}. 
\]
Thus our assumptions on $U$ and $V$ imply that
\begin{align}
\heat X &\leq \frac{C(U+V)W}{V}+C'CWX -\frac{c}{2}\frac{|\cv X|^2}{X}\nonumber\\
&\leq C'CW(1+X)-\frac{c}{2}\frac{|\cv X|^2}{X}. \label{eq:x_heat}
\end{align}

Now for $\lambda=\lambda(c,C) > 0$ to be chosen, we define $Y = X - (\lambda t)(1+X)W$ and compute that
\begin{align*}
\heat Y &= (1-\lambda tW)\heat X-\lambda t(1+X)\heat W\\
&\quad -\lambda(1+X)W+2\lambda t\lp\cv X,\cv W\rp.
\end{align*}
If $t\leq T_1:=1/(\lambda C)$, then $1-\lambda tW\geq0$, so we may apply estimate~\eqref{eq:x_heat} to the first term
on the \textsc{rhs} above. We then use our assumption on $\big|\heat W\big|$ to estimate the second term and apply
Cauchy--Schwarz to the last term, obtaining
\begin{align*}
\heat Y&\leq(1-\lambda tW)\left\{ C'CW(1+X)-\frac{c}{2}\frac{|\cv X|^2}{X}\right\}+C\lambda tW(1+X)\\
&\quad-\lambda W(1+X)+\lambda t\left(\frac{|\cv X|^2}{X}+X|\cv W|^2\right).
\end{align*}
By using our assumption that $|\cv W|^2\leq CW$, we simplify this to
\begin{align*}
\heat Y&\leq W(1+X)\big\{-\lambda+CC'(1-\lambda tW)+2C\lambda t\big\}\\
&\quad+\frac{|\cv X|^2}{X}\big\{-\frac{c}{2}(1-\lambda tW)+\lambda t\big\}.
\end{align*}
Then choosing $\lambda=2C'$ and using our upper bound for $W$, we obtain
\[
\heat Y\leq W(1+X)CC'(-1+4Ct)+\frac{|\cv X|^2}{X}\big\{-\frac{c}{2}+CC'(Cc+2)t\big\}.
\]
The \textsc{rhs} is nonpositive provided that $t\leq T_2:=\frac{1}{4C}$ and
$t\leq T_3:=\frac{c}{2CC'(Cc+2)}$. Thus we choose $T'=\min\{T_1,T_2,T_3\}$.

Finally, we justify applying the weak maximum principle on the noncompact manifold $\mc M$ in the form
detailed in Theorem~12.22 of~\cite{vol-144}. Specifically,
since $W$ is bounded, the assumption that $U(x,t)/V(x,t)\leq C'e^{C' d^2_{g(t)}(x',x)}$ implies easily
that Theorem~12.22 applies to $Y(x,t)-\sup_{y\in\mc M}Y(y,0)$, allowing us to conclude
\[X(x,t) \le \lambda t W(x,t) (1 + X(x,t)) + \sup_y X(y,0),\]
implying that
\[X(x,t) \le \frac{\lambda t}{1-\lambda t W(x,t)}\, W(x,t) + \frac{1}{1-\lambda t W(x,t)}\, \sup_{y\in \mc M} X(y,0).\]
We can decrease $T'$ if necessary to make $\lambda t W(x,t)$ small enough for all $t\in [0,T']$ so that the following holds:
\begin{align*}
X(x,t) &\le \lambda t \Big(1 + 2\lambda t W(x,t)\Big)\, W(x,t) + \Big(1 + 2\lambda t W(x,t)\Big)\, \sup_{y\in \mc M} X(y,0) \\
&= \sup_{y\in \mc M} X(y,0) + W(x,t)\, \Big(\lambda t + 2\lambda^2 t^2 W(x,t)^2 + 2\lambda t\, \sup_{y\in \mc M} X(y,0)\Big)\\
&  \le \sup_{y\in \mc M} X(y,0) + 2\lambda t W(x,t)\, \Big(1 + \sup_{y\in \mc M} X(y,0)\Big).
\end{align*}
This completes the proof.
\end{proof}

\subsection{Main estimates}	\label{main_estimates}
We now establish two pairs of Propositions that provide the key results we need to prove Theorem~\ref{main}.

In Propositions~\ref{lem:Tim-15} and \ref{cor-improvement}, we obtain bounds for $\gamma$, $\chi$, and $\rho$
on a solution that is smooth on a compact time interval $[t_0,t_1]$.\footnote{In \emph{Step 1} of our proof of 
Theorem~\ref{main}, we initially apply Propositions~\ref{lem:Tim-15} and \ref{cor-improvement}
with $t_0=0$. In \emph{Step 2} however, we need to apply them at some $t_0>0$.}
 A strength of these results is that they allow us to extend bounds that hold at $t_0$ to the entire
interval $[t_0,t_1]$, which cannot be taken for granted because $\mc M$ is noncompact; a weakness is that
the constant we obtain for these bounds depends on an upper bound for the full curvature tensor on
$[t_0,t_1]$.

In Propositions~\ref{lem:Tim-16} and \ref{uniform},
we show that if the functions $v_\alpha$ and their derivatives satisfy uniform bounds on an interval $[0,T]$, then
those bounds can be improved, independent of the curvature, at least on an interval $[0,T_*]$, with
$0<T_*\leq T$.

\begin{prop}	\label{lem:Tim-15}
Let $(\mc M,g_\init)$ satisfy the Main Assumptions in Section~\ref{sec:assumptions}.

Suppose a solution $g(t)$ of Ricci flow exists for $[t_0,t_1]$, satisfying the initial bounds
$\rho(x,t_0)\leq C_0C_{\init}$ and 
\begin{equation}	\label{eq-assump}
v_\alpha(x,t_0)>0,\quad
\gamma_\alpha(x,t_0)\leq C_0C_{\init}\,G_\alpha\big(v_\alpha(x,t_0)\big),\quad
\chi_\alpha(x,t_0)\leq C_0C_{\init}\,H_\alpha\big(v_\alpha(x,t_0)\big),
\end{equation}
along with the uniform bound $\sup_{(x,t)\in\mc B\times[t_0,t_1]}|\Rm(x,t)|\leq C_1$,
for some constants $C_0,C_1$.

Then there exists $C'=C'(C_\init,C_0, C_1)$ and $T' = T'(C_\init, C_0, C_1) \in (t_0, t_1]$ such that for all $t\in[t_0,T']$, one has
\begin{subequations}		
\begin{align}
\label{eq-v-lower-15}
v_\alpha(x,t) &\geq \frac{v_\alpha(x,t_0)}{1+C'(t-t_0)},\\
\label{eq-grad-15}
\gamma_\alpha(x,t) &\leq C_0 C_{\init} (1 + C'\,(t-t_0))\,G_\alpha\big(v_\alpha(x,t_0)\big),\\
\label{eq-kappa-15}
\chi_\alpha(x,t) &\leq C_0 C_{\init} (1 + C'(t-t_0))\, H_\alpha\big(v_\alpha(x,t_0)\big), \\
\label{eq-rho-15}
\rho(x,t) &\leq C_0C_{\init}\,(1+C'(t-t_0)).
\end{align}
\end{subequations}
\end{prop}

Because its proof is lengthy, we prove Proposition~\ref{lem:Tim-15} in a series of steps that are contained in
Lemmas~\ref{lemma-gradrm-bnd}--\ref{lemma-rho}.
In the course of the proof, we use the same symbols $C'$ and $T'$ for possibly different constants that depend only on $C_\init$, $C_0$, and $C_1$ --- with $C'$ allowed to grow but remain finite, and $T'$ allowed to shrink but remain positive.
\smallskip

Our first observation is needed because to prove our Main Theorem,
we need to apply Lemma~\ref{lem:Tim-17} in cases
where $V$ may be independent of time, but $\Delta V$ and $|\cv V|^2$ are computed with respect to $g(t)$.

\begin{lemma}\label{lemma-gradrm-bnd}
  Suppose that the assumptions of Proposition \ref{lem:Tim-15} hold.

  Then there are a constant $C'(C_\init, C_0, C_1)$ and time $T'(C_\init, C_0, C_1) \in (t_0, t_1]$
  such that on $[t_0, T']$, we have
  \begin{align*}
    |\nabla \Rm[g(t)]|_{g(t)} &\leq C', \\
    \frac{|\nabla v_\alpha(x, t_0)|^2_{g(t)}}{v_\alpha(x, t_0)^2} &\leq C', \\
    \frac{|\Delta_{g(t)} v_\alpha(x, t_0)|_{g(t)}}{v_\alpha(x, t_0)} &\leq C'.
  \end{align*}
  Note that the final two collections of inequalities can be summarized as
  \begin{align}
    \parexp{v_\alpha(x, t_0)} \leq C'. \label{eq:exp-bnd-v}
  \end{align}
\end{lemma}

\begin{proof}
If $t_0 > 0$, our assumed bound on $|\Rm|$ at time $t = t_0$ and regularity theory for Ricci flow imply the
stated bound for $|\cv \Rm|$. If $t_0=0$, we note that the Main Assumptions outlined in Section~\ref{sec:assumptions}
include an upper bound  for  $|\nabla \Rm|$ at time $t=0$. Then Theorem~14.16 of~\cite{vol-144}) lets
us bound $|\cv \Rm|$ on $[t_0,T']$.

 The subsequent inequalities follow because they hold at time $t_0$, and because $\partial_t g$
  and $\partial_t \Gamma$ are controlled by our bounds on $|\Rm|$ and $|\nabla \Rm|$.
  \end{proof}

\begin{lemma}
\label{lemma-v-lower}
Suppose that the assumptions of Proposition \ref{lem:Tim-15} hold.

Then there exist a constant $C'(C_\init, C_0, C_1)$ and time $T'(C_\init, C_0, C_1) \in (t_0, t_1]$ so that for all $t\in[t_0, T']$, estimate \eqref{eq-v-lower-15}  holds.
\end{lemma}

\begin{proof}
First we claim there exists a $T'$ so that $v_\alpha\geq 0$ on $\mc M\times[t_0,T']$.

\smallskip

The Ricci flow equation restricted to the metric on $\mc{F}_{\alpha}$ is
\[\partial_t (u_{\alpha} g_{\alpha}) = -2\Rc\big|_{\alpha},\]
where $\Rc\big|_{\alpha}$ denotes the Ricci curvature of planes tangent  to $\mc{F}_{\alpha}$. Using
the fact that $g_{\alpha}$ is independent of time, we can rewrite this as
\[\partial_t(\log u_{\alpha}) u_{\alpha} g_{\alpha} = -2\Rc\big|_{\alpha}.\]
Since $|\Rm[g(t)]| \le C_1$ on $[t_0,t_1]$, we get a comparable bound for $|\Rc|$, implying that
\[\big|\partial_t \log u_{\alpha}\big| \le C',\]
where $C' = C'(C_1)$. Hence for all $t\in [t_0,t_1]$, we have
\begin{equation}
\label{eq-lower-u}
u_{\alpha}(x,t) \ge e^{-C'(t-t_0)}\, (a_{\alpha} - \mu_{\alpha} t_0).
\end{equation}

To prove the claim that $v$ remains nonnegative for a short time,
we first show that given any $\delta>0$, we have $v_\alpha\geq-\delta t$ on  a time interval $[t_0,T']$,
where $T'$ could possibly decrease in the proof but is independent of $\delta$. 

Equation~\eqref{eq-lower-u} implies that 
$
u_\alpha
\ge (a_{\alpha} - \mu_{\alpha} t_0)\, \left(1 - C'(t - t_0)\right)
$,
so 
\begin{equation}
\label{eq:coarse-v-estimate}
v_{\alpha}(x,t) = u_\alpha - (a_\alpha - \mu_\alpha t) \ge - C'\,(t - t_0).
\end{equation}
We fix $T'$ so that $T' - t_0 \le (a-\mu_{\alpha} t_1)/(2C_1)$ and
let $\tau \in [t_0,T']$ be arbitrary. Because $C_1 (\tau - t_0)\leq C_1(T' - t_0) \le \tfrac12(a_{\alpha}-\mu t_1)$, we may let
\begin{equation}\label{eq:epsilon-lower-bound}
 \epsilon\in\big(C_1(\tau - t_0),\,a_{\alpha}-\mu_{\alpha} t_1\big)
\end{equation}
 be arbitrary. Then $v_\alpha+\epsilon>0$ on $[t_0,\tau]$, so each function
 $v_{\alpha,\epsilon}:=(v_\alpha+\epsilon)^{-1}$ is well defined
on that time interval. Using that $v_{\alpha}(x,t)$ evolves by
\begin{equation}	\label{eq-v}
\heat  v_{\alpha} =  - \frac{\gamma_{\alpha}}{u_{\alpha}},
\end{equation}
a straightforward computation yields
\[
\heat v_{\alpha,\epsilon} = |\nabla\log v_{\alpha,\epsilon}|^2 v_{\alpha,\epsilon} \, \left(\frac{v_{\alpha}+\epsilon}{u_{\alpha}} - 2\right).
\]
Our choice of $\epsilon$ implies that
\[
  \frac{v_{\alpha}+\epsilon}{u_{\alpha}}
  =   \frac{v_{\alpha}+\epsilon}{v_{\alpha} + a_{\alpha} - \mu_{\alpha} t}
  \le \frac{v_{\alpha} + \epsilon}{v_{\alpha} + a_{\alpha} - \mu_{\alpha} t_1}
  \le 1
\]
for all $t \in [t_0,T']$ and thus that
\[
\heat v_{\alpha,\epsilon} \le - |\nabla \log v_{\alpha,\epsilon}|^2 v_{\alpha,\epsilon} = 
 - \frac{|\nabla v_{\alpha,\epsilon}|^2}{ v_{\alpha,\epsilon}}  .
\]
Let $U(x,t) = v_{\alpha,\epsilon}(x,t)$, $V(x,t) = v_{\alpha,\epsilon}(x,t_0)$, and $W = C'$. Observing that $V$ is independent of time
and using Lemma \ref{lemma-gradrm-bnd},  one sees that
\begin{equation}	\label{eq:exp-bnd-v-eps}
\begin{split}
\frac{\left|\heat V\right|}{V} + \frac{|\nabla V|^2}{V^2} &\le
\frac{|\Delta_{g(t)} v_{\alpha}(x,t_0)|}{v_{\alpha}(x,t_0) + \epsilon} + \frac{|\nabla v_{\alpha}(x,t_0)|^2_{g(t)}}{(v_{\alpha}(x,t_0) + \epsilon)^2} \\
&\le C'.
\end{split}
\end{equation}
Note in particular that~\eqref{eq:exp-bnd-v-eps} is independent of $\epsilon$. For a sufficiently short time,
$v_{\alpha,\epsilon} \le (\epsilon - C'(\tau-t_0))^{-1} < \infty$. So  $U$ is bounded in space, and the bound \eqref{eq:exp-bnd-v-eps} implies $|\nabla \log V|$ is bounded, so $V$ decays at most exponentially.  Thus, Lemma~\ref{lem:Tim-17} can be applied to $U, V$, and $W$  as defined above to conclude that
\[
v_{\alpha,\epsilon}(x,t) \le  (1 + C'\,(t - t_0))\, v_{\alpha,\epsilon}(x,t_0),\qquad t\in[t_0,\tau],
\]
where $C' = C'(C_\init, C_0,C_1)$ is independent of $\epsilon$.
Letting $\epsilon\searrow C_1(\tau - t_0)$, which is the lower bound imposed by \eqref{eq:epsilon-lower-bound}, we find that
\[
v_\alpha+C_1(\tau - t_0)\geq\frac{C_1(\tau-t_0)}{1+C'\,(\tau-t_0)}\geq\frac{C_1(\tau- t_0)}{1 + C'\,(t_1-t_0)}.
\]
Because $\tau\in[t_0,T']$ is arbitrary, this implies that 
\[
v_\alpha\geq-\left(1-\frac{1}{1+C'(t_1-t_0)}\right)C_1\,(t - t_0)
\]
 for all $t\in[t_0,T']$, which improves~\eqref{eq:coarse-v-estimate} by a fixed factor. Repeating this
bootstrap  argument $k$ times (which can be done without changing $T'$), where
\[
\left(1-\frac{1}{1+C'(t_1-t_0)}\right)^k C_1\leq\delta,
\]
 proves that $v_\alpha\geq-\delta (t-t_0)$ on $[t_0,T']$. Because $\delta>0$ is arbitrary and $T'$ is independent of $\delta$, it follows
 that $v_{\alpha}\geq0$ on $[t_0,T']$, as claimed.

\smallskip

We next prove a better quantitative lower bound for $v_{\alpha}$, as long as $v_{\alpha}(x,t) \geq 0$
holds, that is, for $t\in [t_0,T']$, where $T'$ is some possibly smaller time $T'(C_{\init},C_0,C_1)$.

The method is very close to that used in the proof of
the claim that $v_\alpha\geq0$, so we avoid unnecessary repetition. Let $\epsilon\in(0,a_\alpha-\mu_\alpha t_1)$ be arbitrary, and again let $v_{\alpha,\epsilon} := (v_{\alpha} + \epsilon)^{-1}$.   
Note that in contrast to the previous argument, where~\eqref{eq:epsilon-lower-bound} is needed,
we have proven that $v_{\alpha} \geq 0$ above, hence we know that $v_{\alpha,\epsilon}$ is well-defined and
bounded by $\epsilon^{-1}$ for all $\epsilon > 0$. Then as in the arguments above, we find that
\[
\heat v_{\alpha,\epsilon} \leq-\frac{|\cv v_{\alpha,\epsilon}|^2}{v_{\alpha,\epsilon}}.
\]
Now let  $U(x,t) = v_{\alpha,\epsilon}(x,t)$,
$V(x,t) = v_{\alpha,\epsilon}(x,t_0)$, and $W = C'$.
Again using equation \eqref{eq:exp-bnd-v-eps} and the fact  that $U$ is bounded, we can apply Lemma \ref{lem:Tim-17} to obtain
\[
v_{\alpha,\epsilon}(x,t) \le (1+C'(t - t_0))\, v_{\alpha,\epsilon}(x,t_0),
\]
where $C' = C'(C_\init, C_0,C_1)$ is  independent of $\epsilon$. We let $\epsilon\searrow0$ to conclude that
\[
\frac{v_\alpha(x,t_0)}{1+C'(t-t_0) } \le v_\alpha(x,t)
\]
on $\mc M\times[t_0, T']$.
\end{proof}

\begin{lemma}
\label{lemma-gradient}
Under the assumptions of Proposition~\ref{lem:Tim-15}, 
there exist $C'(C_\init, C_0, C_1)$ and $T' = T'(C_\init, C_0, C_1) \in (t_0, t_1]$
such that for all $t \in [t_0, T']$, estimate \eqref{eq-grad-15}  holds.
\end{lemma}

\begin{proof}

Since we have \eqref{eq-lower-u},
we can find $T'$ sufficiently small so that
\begin{equation} \label{eq-lower-first-u}
u_{\alpha}(x,t) \ge \frac 12 \inf_{\mc{M}} u_{\alpha}(\cdot,t_0) > 0 \qquad \mbox{ for } \quad t\in[t_0,T'].
\end{equation}
Recalling that\footnote{See~\eqref{eq:KN-convention} for our normalization of the Kulkarni--Nomizu product $\KN$.}
 $\Rm[g_{\mc F_\alpha}]=c_\alpha\,g_{\mc F_\alpha}\KN g_{\mc F_\alpha}$ and using formula~\eqref{eq:rmmultiple},
which we derive in Lemma~\ref{claim:rmmultiple} in Appendix~\ref{multiply-warped},
one sees easily that the bound $|\Rm[g(t)]| \le C_1$ on $[t_0,t_1]$ implies that
\[
 \big| c_\alpha u_\alpha^{-1} - \frac12 |\cv(\log u_\alpha^{1/2})|^2\big| \leq C_1.
 \]
Combining this with~\eqref{eq-lower-first-u} implies the existence of $C'=C'(\inf_{\mc{M}}\, u_{\alpha}(\cdot,t_0),\,C_1)$ such that
\begin{equation}	\label{eq-gamma-est1}
\frac{\gamma_{\alpha}}{u_{\alpha}^2} = 4 |\cv(\log u_\alpha^{1/2})|^2 \le C'
\end{equation}
for all $(x,t)\in \mc{M}\times [t_0,T']$.
Using \eqref{eq-nabla-v} and \eqref{eq-gamma-est1} yields
\begin{equation}
\label{eq-nabla-v1}
\heat\gamma_{\alpha} \le  C'\, \gamma_{\alpha} -\frac 12 \, \frac{|\nabla \gamma_{\alpha}|^2}{\gamma_{\alpha}}.
\end{equation}
We apply Lemma \ref{lem:Tim-17} to \eqref{eq-nabla-v1} with $U(x,t) = \gamma_{\alpha}(x,t)$,
$V(x,t) = G_{\alpha}\big(v_{\alpha}(x,t_0)\big)$, and $W(x,t) = C'$. To see that all assumptions of Lemma~\ref{lem:Tim-17}
are satisfied, we need to check that $\parexp{V}$ is bounded by $C'$. Indeed, by \eqref{eq:easycalc} and \eqref{eq:exp-bnd-v} we have
\begin{equation}\label{eq:exp-bnd-v-gamma}
\parexp{V} \le \polyd{G_{\alpha}}\, \parexp{v_{\alpha}(x,t_0)} \le C'.
\end{equation} 
We also need to check that $\frac{U(x,t)}{V(x,t)} \le C'\, e^{C' d^2_{g(t)}(x,x_0)}$ for $t\in [t_0,T']$, where $x_0$ is some fixed point in $\mc B$. Indeed, since $\nabla u_{\alpha} = \nabla v_{\alpha}$, \eqref{eq-gamma-est1} implies that for every $t\in [t_0,T']$, the function $u_{\alpha}(\cdot,t)$ grows at most exponentially in space and thus, for every $t\in [t_0,T']$, $U(x,t) = \gamma_{\alpha}(x,t)$  grows at most exponentially in space as well.  On the other hand, by \eqref{eq:exp-bnd-v-gamma}, $|\nabla\log V|$ is bounded, so $V(x,t)$ has at most exponential decay. Hence,
\[\frac{U}{V}(x,t) \le C' e^{C' d_t(x,x_0)} \qquad \mbox{on}\,\,\,\, \mc{M}\times[t_0,t_1],\]
as desired.

We can finally apply Lemma \ref{lem:Tim-17} as indicated above to conclude that for $t\in [t_0,T']$, we have
\begin{equation}
\label{eq-nabla-bound-part}
\gamma_{\alpha}(x,t) \le (1 + C'\, (t - t_0)) \, C_0 C_{\init} G_{\alpha}(v_{\alpha}(x,t_0)),
\end{equation}
as desired.
\end{proof}
\medskip

\begin{lemma}
\label{lemma-kappa}
Under the assumptions of Proposition~\ref{lem:Tim-15}, there exist a constant $C'=C'(C_\init,C_0, C_1)$  and a time
$T' = T'(C_{\init},C_0,C_1) \in (t_0,t_1]$ such that for all times $t\in [t_0,T']$, estimate \eqref{eq-kappa-15} holds.
\end{lemma}

\begin{proof}

Assume $T' \in (t_0,t_1]$ is chosen so that both~\eqref{eq-v-lower-15}
and~\eqref{eq-grad-15} hold on $[t_0,T']$.

Recall that in Lemma \ref{lemma-ev-eq}, we compute that $\chi_{\alpha}$ evolves by
\[
\heat\chi_{\alpha} \le  C_N L \chi_{\alpha} +
C_N  L \sum_{\beta=1}^A \left(\frac{\gamma_{\beta}}{u_{\beta}^2}\right)\, \gamma_{\alpha} -
\frac12\, \frac{|\nabla\chi_{\alpha}|^2}{\chi_{\alpha}},
\]
where $L := \sum_{\beta=1}^A \frac{\gamma_{\beta}}{u_{\beta}^2} +
\sum_{\beta=1}^A \frac{\chi_{\beta}^{1/2}}{u_{\beta}} + \rho^{1/2}$, $\rho := |\Rm_{\mc{B}}|^2$, and  $C_N$
depends only on the dimension vector $\vec N=(n,n_\alpha)$.
The curvature bound $|\Rm[g(t)]| \le C_1$  and~\eqref{eq-gamma-est1} imply that $|L| \le C'$ on $\mc{M}\times[t_0,T']$.
Hence we have
\[\heat \chi_{\alpha} \le  C' \chi_{\alpha} +
C' \left(\sum_{\beta=1}^A\frac{\gamma_{\beta}}{u_{\beta}^2}\right) \gamma_{\alpha} -
\frac12\, \frac{|\nabla\chi_{\alpha}|^2}{\chi_{\alpha}}.\]
By \eqref{eq-v-lower-15} and \eqref{eq-grad-15}, we see that for $t\in [t_0,T']$,
\[\left(\sum_{\beta=1}^A\frac{\gamma_{\beta}}{u_{\beta}^2}\right) \gamma_{\alpha} \le
C' \,\sum_{\beta=1}^A\frac{G_{\beta}(v_{\beta}(x,t_0))}{v_{\beta}^2(x,t_0)} \, G_{\alpha}(v_{\alpha}(x,t_0)).\]

Let $U(x,t) = \chi_{\alpha}(x,t)$, $V(x,t) =
\sum_{\beta=1}^A\frac{G_{\beta}(v_{\beta}(x,t_0))}{v_{\beta}^2(x,t_0)}\, G_{\alpha}(v_{\alpha}(x,t_0)) = H_{\alpha}(v_{\alpha}(x,t_0))$, and $W = C'$. We verify that $V(x,t)$ satisfies the hypotheses of Lemma \ref{lem:Tim-17}.
By using estimate~\eqref{eq:easycalc}, we obtain
\begin{align*}
\parexp{V} &\le
\Big| \sum_{\beta=1}^A \frac{G_{\beta}(v_{\beta}(x,t_0))}{v_{\beta}^2(x,t_0)} \Big|_{2,\exp}
 +\; \parexp{G_{\alpha}(v_{\alpha}(x,t_0))} \\
&\le 2\sum_{\beta=1}^A
\Big|\frac{G_{\beta}(v_{\beta}(x,t_0))}{v_{\beta}^2(x,t_0)}\Big|_{2,\exp}
+\,\parexp{G_{\alpha}(v_{\alpha}(x,t_0))}.
\end{align*}
Note that by \eqref{eq:easycalc} and \eqref{eq:exp-bnd-v}, we also have
\[\parexp{G_{\alpha}(v_{\alpha}(x,t_0))} \le \polyd{G_{\alpha}}^2\, \parexp{v_{\alpha}(x,t_0)} \le \bar{C}_{\alpha}^2 C'.\]
Moreover, we may regard $\frac{G_{\beta}(v_{\beta}(x,t_0))}{v_{\beta}^2(x,t_0)}$ as a composition of functions
$\vp_{\beta}(s) := \frac{G_{\beta}(s)}{s^2}$ and $v_{\beta}(x,t_0)$. Then using \eqref{eq:easycalc} again, we obtain
\[
\Big|\frac{G_{\beta}(v_{\beta}(x,t_0))}{v_{\beta}^2(x,t_0)}\Big|_{2,\exp} \le \polyd{\vp_{\beta}}^2\, \parexp{v_{\beta}(x,t_0)}.
\]
By \eqref{eq:exp-bnd-v}, we have $\parexp{v_{\beta}(x,t_0)} \le C'$. It is easy to see that
\[\frac{s \vp_{\beta}'}{\vp_{\beta}} \le \frac{s |G_{\beta}'(s)|}{G_{\beta}(s)} + 2 \le \bar{C}_{\beta} + 2\]
and
\[\frac{s^2 |\vp_{\beta}''|}{\vp_{\beta}(s)} \le \frac{s^2 |G_{\beta}''(s)|}{G_{\beta}(s)} + 4\, \frac{s |G_{\beta}'(s)|}{G_{\beta}(s)} + 6 \le 5\,\bar{C}_{\beta} + 6.\]
These imply that
\[
\Big|\frac{G_{\beta}(v_{\beta}(x,t_0))}{v_{\beta}^2(x,t_0)}\Big|_{2,\exp} \le C',
\]
and hence that
\begin{equation}
\label{eq-exp-decay}
\parexp{V} \le C'.
\end{equation}

Recall that in this proof, we choose $U = \chi_{\alpha} = |\nabla \nabla u_\alpha|_{g_\mc B}$. 
Our assumption that the curvature is bounded by $C_1$ implies in particular (by Remark \ref{rem-curv}) that $|u^{-1}\nabla \nabla u_\alpha - 1/2 u_\alpha^{-2}\nabla u_\alpha \otimes \nabla u_\alpha|$ is bounded by $C'$. Then \eqref{eq-gamma-est1} implies firstly that $|u^{-1}_\alpha\nabla \nabla u_\alpha| < C'$, and secondly that $u_\alpha$ grows at most exponentially in space, so that $\chi_\alpha$ grows at most exponentially in space.  On the other hand, by \eqref{eq-exp-decay}, we have that $V(x,t)$ decays at most exponentially in space.  These two estimates yield the bound $\frac{U}{V} \le C' e^{C'\, d_{g(t)}(x,x_0)}$. 

We can now apply Lemma \ref{lem:Tim-17} to our choice of $U(x,t)$, $V(x,t)$, and $W(x,t)$ to conclude
\[\chi_{\alpha}(x,t) \le C_0 C_{\init}\, (1+C'\,(t-t_0)) H_{\alpha}(v_{\alpha}(x,t_0)),\]
where we use the initial condition that $\chi_\alpha(x,t_0)\leq C_0C_{\init}\,H_\alpha\big(v_\alpha(x,t_0)\big)$.
\end{proof}

\begin{lemma}
\label{lemma-rho}
Under the assumptions of Proposition~\ref{lem:Tim-15}, there exist a constant $C'=C'(C_\init,C_0, C_1)$  and a 
time $T' = T'(C_{\init},C_0,C_1) \in (t_0,t_1]$ such that for all times $t\in [t_0,T']$, estimate \eqref{eq-rho-15} holds.
\end{lemma}

\begin{proof}
Recall that in Lemma \ref{lemma-ev-eq}, we compute that $\rho(x,t) = |\Rm[g_{\mc{B}}](x,t)|$ evolves by 
\[\heat \rho \le  C_N L^3 - \frac{|\nabla\rho|^2}{\rho},\]
where $L = \sum_{\beta=1}^A \frac{\gamma_{\alpha}}{u_{\alpha}^2} +
\sum_{\beta=1}^A \frac{\chi_{\alpha}^{1/2}}{u_{\alpha}} + \rho^{1/2}$,
and $C_N$ depends only on the dimension vector $\vec N=(n,n_\alpha)$.
As in the proof of Lemma \ref{lemma-kappa}, we conclude that $|L| \le C'$ on $\mc{M}\times[t_0,t_1]$ and hence that
\[
\heat \rho \le  C_* C'\, L^2 - \frac{|\nabla\rho|^2}{\rho}.\]
By Lemma \ref{lemma-gradient} and Lemma \ref{lemma-kappa}, there exist constants $C'$ and $T'\in (t_0,t_1]$ so that for all $t\in [t_0,T']$, we have
\begin{align*}
L &\le C'\, \left(\sum_{\alpha=1}^A \frac{G_{\alpha}(v_{\alpha}(x,t_0))}{v_{\alpha}^2(x,t_0)} +
\sum_{\alpha=1}^A \frac{H_{\alpha}^{1/2}(v_{\alpha}(x,t_0))}{v_{\alpha}(x,t_0)} + \rho^{1/2}\right) \\
&\le C'(\bar{C}_\alpha+ \rho^{1/2}),
\end{align*}
where $\bar{C}$ is a bound on
$\sup_{s_{\alpha}\in \mathbb{R}_+}\left(\sum_{\alpha=1}^A\frac{G_{\alpha}(s_{\alpha})}{s_{\alpha}^2}
+ \sum_{\alpha=1}^A \frac{H_{\alpha}(s_{\alpha})}{s_{\alpha}}\right)$, \emph{i.e.,} a uniform constant. Hence,
\[\heat \rho \le C_*C'\, (\rho + 1)  -  \frac{|\nabla\rho|^2}{\rho}.\]

We apply Lemma \ref{lem:Tim-17} with $U(x,t) = \rho(x,t)$, $V(x,t) = 1$, and $W(x,t) = C'$ to conclude that for all $t\in [t_0,T']$,
we have
\[\rho(x,t) \le C_0 C_{\init} (1+C'(t-t_0)).\]
\end{proof}

Combining Lemmas~\ref{lemma-gradrm-bnd}--\ref{lemma-rho} completes the proof of Proposition~\ref{lem:Tim-15}.
\medskip

Recall that the estimates~\eqref{eq-grad-15} and \eqref{eq-kappa-15} for $\gamma_\alpha(x,t)$ and $\chi_\alpha(x,t)$,
respectively, that we prove in Proposition~\ref{lem:Tim-15} have $v_\alpha(x,t_0)$ on the \textsc{rhs}.
Our next result improves those by substituting $v_\alpha(x,t)$ for $v_\alpha(x,t_0)$.

\begin{prop}	\label{cor-improvement}
Suppose the assumptions of Proposition \ref{lem:Tim-15} hold.

Then there exists $C' = C'(C_{\init},C_0,C_1,t_1-t_0)$  so that we have
\[\gamma_{\alpha}(x,t) \le  (1+ C'(t - t_0))\, C_0 C_{\init} G_{\alpha}(v_{\alpha}(x,t)),\]
\[\chi_{\alpha}(x,t) \le (1+ C'(t - t_0))\, C_0 C_{\init} H_{\alpha}(v_{\alpha}(x,t)),\]
for all $t\in [t_0,T']$, where $T'$ is the same as in Proposition \ref{lem:Tim-15}.
\end{prop}

\begin{proof}
By the chain rule, we have
\[\partial_t \big(G_{\alpha}(v_\alpha)\big) = G_{\alpha}' (v_\alpha)\,\partial_t v_\alpha,\]
which implies that
\begin{equation}	\label{eq-dt-G}
\begin{split}
\left|\partial_t G_{\alpha}(v_{\alpha}(x,t))\right|
&= \left|\frac{v_{\alpha}(x,t) G_{\alpha}'(v_{\alpha}(x,t))}{G_{\alpha}(v_{\alpha}(x,t))}\right|
\frac{G_{\alpha}(v_\alpha(x,t)}{v_{\alpha}(x,t)}
\left| \partial_t v_\alpha(x,t) \right| \\
&\le \bar{C}_{\alpha} G_{\alpha}(v_{\alpha}(x,t)) \, \left( \frac{|\nabla^2 v_{\alpha}(x,t)|}{v_{\alpha}(x,t)} + \frac{|\nabla v_{\alpha}(x,t)|^2}{v_{\alpha}^2(x,t)}\right).
\end{split}
\end{equation}
By \eqref{eq-v-lower-15} and \eqref{eq-grad-15}, which hold for $t\in [t_0,T']$, we have
\begin{equation*}
\begin{split}
|\nabla v_{\alpha}|^2(x,t) &\le (1+C'\,(t-t_0)) C_0 C_{\init} \bar{C}_{\alpha} v_{\alpha}^2(x,t_0) \\
&\le  (1+C'\,(t-t_0))^3 C_0 C_{\init} \bar{C}_{\alpha} v_{\alpha}^2 (x,t)
\end{split}
\end{equation*}
for all $t\in [t_0,T']$. This yields
\begin{equation}
\label{eq-bound-1}
\frac{|\nabla v_{\alpha}|^2}{v_{\alpha}^2} \le (1+C'\,(t-t_0))^3 C_0 C_{\init} \bar{C}_{\alpha} 
\qquad \mbox{for}\quad t\in [t_0,T'].
\end{equation}

To bound $\frac{|\nabla^2 v_{\alpha}(x,t)|}{v_{\alpha}(x,t)}$, we note that by \eqref{eq-v-lower-15} and \eqref{eq-kappa-15},
we have
\begin{equation*}
\begin{split}
\chi_{\alpha}(x,t) &\le (1 + C'\,(t - t_0)) C_0 C_{\init} H_{\alpha}(v_{\alpha}(x,t_0)) \\
&\le (1 + C'\,(t - t_0)) C_0 C_{\init} \bar{C}_{\alpha} v_{\alpha}^2(x,t_0) \\
&\le (1 + C'\,(t - t_0))^3 C_0 C_{\init} \bar{C}_{\alpha} v_{\alpha}(x,t)^2,
\end{split}
\end{equation*}
implying that
\[\frac{|\nabla^2 v_{\alpha}|^2}{v_{\alpha}^2} \le \,\Big(1+ C'\,(t-t_0)\Big)^3 C_0 C_{\init}\bar{C}_{\alpha}, \qquad \mbox{for} \,\,\,\, t\in [t_0,T'],\]
where $\bar{C}_{\alpha}$  is a uniform constant. 
Combining this estimate with \eqref{eq-bound-1} and \eqref{eq-dt-G} yields
\[\left|\partial_t \log G_{\alpha}(v_{\alpha}(x,t))\right| \le (1+ C'\,(t-t_0))^3\, C_0 C_{\init} \bar{C}_{\alpha},\]
and hence
\begin{equation}
\label{eq-equiv-G}
G_\alpha(v_\alpha(x,t))\leq(1+C'(t-t_0))\,G_\alpha(v_\alpha(x,t_0))
\qquad \mbox{for all} \,\,\,\, t\in [t_0,T'].
\end{equation}
We combine \eqref{eq-grad-15} and \eqref{eq-equiv-G} to conclude that for all $t\in [t_0,T']$, we have
\[\gamma_{\alpha}(x,t) \le  (1 + C'\,(t - t_0)) C_0 C_{\init}\,G_{\alpha}(v_{\alpha}(x,t)).\]
Finally, using \eqref{eq-v-lower-15}, \eqref{eq-kappa-15}, and \eqref{eq-equiv-G} yields
\[\chi_{\alpha}(x,t) \le (1 + C'\,(t - t_0)) C_0 C_{\init} H_{\alpha}(v_{\alpha}(x,t))\]
for all $t\in [t_0,T']$, as claimed.
\end{proof}
\medskip

We now prove our second pair of Propositions, which provide control of the curvatures by a constant that
depends only on the initial data. Specifically, in contrast to Propositions~\ref{lem:Tim-15} and~\ref{cor-improvement},
the constant we obtain below is independent of the bound $\sup_{(x,t)\in\mc B\times[t_0,t_1]}|\Rm(x,t)|\leq C_1$.

\begin{prop}	\label{lem:Tim-16}
Let $(\mc M,g_\init)$ satisfy the Main Assumptions in Section~\ref{sec:assumptions}.

Suppose a solution $g(t)$ of Ricci flow exists for $[0,T]$ and satisfies
\begin{equation}
\label{eq-initial-16}
\rho \le C_{\init},\quad
\gamma_{\alpha}(x,0) \le C_{\init} G_{\alpha}(v_{\alpha}(x,0)), \quad \chi_{\alpha}(x,0) \le C_{\init}\,H_{\alpha}(v_{\alpha}(x,0)).
\end{equation}
Furthermore, suppose that for $t\in [0,T]$, we have $v_\alpha(x,0)>0$ and
\begin{subequations}	\label{eq:open-closed}
\begin{align}
\gamma_\alpha(x,t)&\leq 2\,C_{\init}\,G_\alpha\big(v_\alpha(x,t)\big),\\
\chi_\alpha(x,t)&\leq 2\, C_{\init}\,H_\alpha\big(v_\alpha(x,t)\big),\\
\rho(x,t)&\leq 2\,C_\init.
\end{align}
\end{subequations}

Then there exists $C_*$ depending only on $C_\init$ and $\vec N=(n,n_\alpha)$ 
such that for $t\in[0,\min\{T,C_*^{-1}\}]$,
one has the bounds
\begin{subequations}
\begin{align*}
\gamma_\alpha(x,t)&\leq C_{\init} G_{\alpha}(v_{\alpha}(x,t)) \,\, \Big(1 + C_* t E_{\alpha} \big(v_{\alpha}(x,t)\big)\Big), \\
\chi_\alpha(x,t)&\leq C_{\init}\,H_\alpha\big(v_\alpha(x,t)\big)\, \Big(1 + C_*\, t\Big),\\
\rho(x,t)&\leq C_\init\, (1 + C_* \, t),
\end{align*}
where
\[
E_\alpha\big(v_\alpha(x,t)\big):=\frac{G_\alpha\big(v_\alpha(x,t)\big)}{v_\alpha(x,t)^2}.
\]
\end{subequations}
\end{prop}

\begin{proof}
The proof is very similar to the proof of Proposition \ref{lem:Tim-15}. We let $C_*' = C_*'(C_{\init})$ be a uniform constant 
that may increase from line to line, whereas $C_*$ is the final constant that appears in the statement above.

To obtain the desired bound for $\gamma_\alpha$, we recall estimate~\eqref{eq-nabla-v},
\[\heat\gamma_\alpha
 \le  6\, \left(\frac{\gamma_{\alpha}}{v_{\alpha}^2}\right)\, \gamma_{\alpha} - \frac12 \frac{|\nabla\gamma_{\alpha}|^2}{\gamma_{\alpha}}.\]
By \eqref{eq:open-closed},  we have
\[\gamma_{\alpha}(x,t) \le 2\,C_{\init} G_{\alpha}(v_{\alpha}(x,t)),\]
and hence
\[\heat\gamma_{\alpha} \le  C_*'\, E_{\alpha}(v_{\alpha})\, \gamma_{\alpha} - \frac12 \frac{|\nabla\gamma_{\alpha}|^2}{\gamma_{\alpha}}.\]

\smallskip

Our goal is to apply Lemma \ref{lem:Tim-17} to $U(x,t) = \gamma_{\alpha}(x,t)$, $V(x,t) = G_{\alpha}(v_{\alpha}(x,t))$, and $W(x,t) = E_{\alpha}(v_{\alpha}(x,t))$. In order to do this, we need to verify that 
the hypotheses of Lemma \ref{lem:Tim-17} are satisfied. By \eqref{eq:easycalc}, we have
\[\parexp{G_{\alpha}(v_{\alpha})} \le \polyd{G_{\alpha}}^2\, \parexp{v_{\alpha}} \le C_*'\, \parexp{v_{\alpha}}\]
and
\begin{equation}
\label{eq-parexp-v-100}
\parexp{v_{\alpha}} = 2\frac{\gamma_{\alpha}}{v_{\alpha}^2} \le C_*'\, E_{\alpha}(v_{\alpha}).
\end{equation}
Thus,
\[\parexp{G_{\alpha}(v_{\alpha})} \le C_*' \, E_{\alpha}(v_{\alpha}).\]

\smallskip

On the other hand, $E_{\alpha}(v_{\alpha}) \le \|G_\alpha\|_{\mc G}\leq  \bar{C}_{\alpha}$ and 
\[\parexp{E_{\alpha}(v_{\alpha})} \le \polyd{E_{\alpha}}^2 \parexp{v_{\alpha}}.\]
It is easy to see that $E_{\alpha}(s) = \frac{G_{\alpha}(s)}{s^2} \le C$ satisfies $\polyd{E_{\alpha}} \le C_*'$ and that $E_{\alpha}(v_{\alpha}) \le C_*'$ is bounded. Hence we have
\[\parexp{E_{\alpha}(v_{\alpha})} \le C_*',\]
implying that 
\[\left|\heat E_{\alpha}(v_{\alpha})\right| + |\nabla E_{\alpha}(v_{\alpha})|^2 \le C_*'\, E_{\alpha}(v_{\alpha}).\]

We can now apply Lemma \ref{lem:Tim-17} as indicated above to conclude that there exists 
$C_*' = C_*'(C_{\init},N)$ and $T_* = T_*(C_{\init},N)$ so that for all $t\in [0,\min\{T_*,T\}]$, we have
\begin{equation}
\begin{split}
\frac{\gamma_{\alpha}(x,t)}{G_{\alpha}(v_{\alpha}(x,t))} &\le C_{\init} + C_*' t\, (1+ C_{\init}) \, E_{\alpha}\big(v_{\alpha}(x,t)\big) \\
&\le C_{\init}\, \Big(1 + C_*\, t\, E_{\alpha}\big(v_{\alpha}(x,t)\big)\Big).
\end{split}
\end{equation}
This yields
\[\gamma_{\alpha}(x,t) \le C_{\init} \, G_{\alpha}(v_{\alpha}(x,t))\,
\Big(1 + C_*\, t\, E_{\alpha}\big(v_{\alpha}(x,t)\big)\Big)\]
for $t\in [0,\min\{T,T_*\}]$, as claimed. 

\medskip

To obtain the bound for $\chi_{\alpha}$, we recall estimate~\eqref{kappa-evolution}, namely
\[\heat\chi_{\alpha} \le  C_N\, L \chi_{\alpha}
+ C_N \, L \, \sum_{\beta=1}^A \frac{\gamma_{\beta}}{u_{\beta}^2}\, \gamma_{\alpha} - \frac12\, \frac{|\nabla\chi_{\alpha}|^2}{\chi_{\alpha}},\]
where $L = \sum_{\beta=1}^A \frac{\gamma_{\beta}}{u_{\beta}^2}
+ \sum_{\beta=1}^A \frac{\chi_{\beta}^{1/2}}{u_{\beta}} + \rho^{1/2}$,
and the constant $C_N$ depends only on $\vec N=(n,n_\alpha)$.
By the assumptions in~\eqref{eq:open-closed}, we have $L \le C_*'$ and
\[
\left(\sum_{\beta=1}^A \frac{\gamma_{\beta}}{u_{\beta}^2}\right)\gamma_{\alpha} \le C_*'\, H_{\alpha}(v_{\alpha}),\]
where $C_*'$ is a uniform constant depending only on $C_{\init}$. Thus,
\[\heat \chi_{\alpha} \le  C_N\,  \chi_{\alpha} + C_N  \, H_{\alpha}(v_{\alpha}) - \frac12\, \frac{|\nabla\chi_{\alpha}|^2}{\chi_{\alpha}}.\]

\smallskip

We want to apply Lemma \ref{lem:Tim-17} to $U(x,t) = \chi_{\alpha}(x,t)$, $V(x,t) = H_{\alpha}(v_{\alpha}(x,t))$, and $W(x,t) = C_*'$. By \eqref{eq:easycalc}, the fact that $\polyd{G_{\alpha}} + \polyd{E_{\alpha}} \le C_*'$,
\eqref{eq:open-closed}, and~\eqref{eq-parexp-v-100}, we have
\begin{equation*}
\begin{split}
\parexp{H_{\alpha}(v_{\alpha})} &\le \parexp{\sum E_{\beta}(v_{\beta})} + \parexp{G_{\alpha}(v_{\alpha})} \\
&\le 2\sum \parexp{E_{\beta}(v_{\beta})} + \parexp{G_{\alpha}(v_{\alpha})} \\
&\le C_*'\, \parexp{v_{\alpha}} = 2C_*'\, \frac{|\nabla v_{\alpha}|^2}{v_{\alpha}^2} \\
&\le C_*' \frac{G_{\alpha}(v_{\alpha})}{v_{\alpha}^2} \le C_*'.
\end{split}
\end{equation*}
By Lemma \ref{lem:Tim-17} applied as above, there exist $C_*'= C_*'(C_{\init}, C_N)$ and
$T_* = T_*(C_{\init}, C_N,)$ such that for all $t\in [0,\min\{T,T_*\}]$, we have 
\[\frac{\chi_{\alpha}(x,t)}{H_{\alpha}(v_{\alpha}(x,t))} \le C_{\init} + C_*' \, t\, (1 + C_{\init}),\]
which implies that
\[\chi_{\alpha}(x,t) \le C_{\init}\, H_{\alpha}(v_{\alpha}(x,t))\,( 1 + C_* t).\]
\medskip

To obtain the desired bound for $\rho(x,t)$, we recall estimate~\eqref{eq-rho},
\[\heat \rho \le   C_N L^3 -  \frac{|\nabla\rho|^2}{\rho},\]
where $L$ is as above,
and $C_N$ depends only on $ \vec N=(n,n_\alpha)$. Hence,
\[\heat \rho \le  C'_* L^2 - \frac{|\nabla\rho|^2}{\rho}.\]
As in Lemma \ref{lemma-rho}, using \eqref{eq:open-closed}, we obtain
\[\heat \rho \le C_*' \, (\rho + 1) - c_N\, \frac{|\nabla\rho|^2}{\rho}.\]

\smallskip

We take $U(x,t) = \rho(x,t)$, $V(x,t) = 1$ and $W(x,t) = C_*'$ and apply Lemma \ref{lem:Tim-17}.  It gives us the existence of
$C_*' = C_*'(C_{\init},T,C_N)$ and $T_* = T_*(C_{\init},T,C_N)$ such that for all $t\in [0,\min\{T,T_*\}]$, we have
\[\rho(x,t) \le C_{\init} + C_*'\, t\, (1 + C_{\init}),\]
implying that
\[\rho(x,t) \le C_{\init}\, (1 + C_*\, t),\]
for all $t\in [0,\min\{T,T_*\}]$, where $C_*$ is a uniform constant.

\medskip

Finally, we increase $C_*$ if necessary so that $C_*T_*\geq1$. This concludes the proof of
Proposition~\ref{lem:Tim-16}.
\end{proof}

Next, we improve Proposition~\ref{lem:Tim-15} by showing
in a precise sense that the quantities $v_\alpha$ are
uniformly equivalent for $t\in[0,\min\{T,C_*^{-1}\}]$, where $[0,T]$ is the time interval on which the
hypotheses of Proposition~\ref{lem:Tim-16} hold. As we note in the Introduction, Theorem~\ref{asymptotics}
follows easily from the arguments that we use to prove Proposition~\ref{uniform}.

\begin{prop}	\label{uniform}
Let $(\mc M,g_\init)$ satisfy the Main Assumptions  in Section~\ref{sec:assumptions},
and let $T$ and $C_*$ be as in the statement of Proposition~\ref{lem:Tim-16}.
Then the quantities $v_\alpha$ are uniformly equivalent:   for $t\in[0,\min\{T,C_*^{-1}\}]$, we have
\[
\frac{1}{C_*}v_\alpha(x,t)\leq v_\alpha(x,0)\leq C_* v_\alpha(x,t).
\]
Furthermore, if in addition to the Main Assumptions, it is also true that
\begin{equation}	\label{eq:decay}
\frac{G_\alpha(s)}{s^2}=o(1;\,s\searrow0),
\end{equation}
then for all $t\in[0,\min\{T,C_*^{-1}\}]$,
\[
v_\alpha(x,t)=\big(1+o(1;\,v_\alpha(x,0)\searrow0)\big)\,v_\alpha(x,0).
\]
\end{prop}

\begin{proof}
It follows easily from~\eqref{eq:evolve-fiber} that there exists $C_N=C_N(n,n_\alpha)$  such that
\begin{equation}	\label{eq:v-evolution}
\partial_t v_\alpha = \Delta v_\alpha-\frac{|\cv v_\alpha|^2}{a_\alpha-\mu_\alpha t + v_\alpha}
\leq C_N\left( \frac{|\cv\cv v_\alpha|}{v_\alpha} +
\frac{v_\alpha}{a_\alpha-\mu_\alpha t+v_\alpha}\frac{|\cv v_\alpha|^2}{v_\alpha^2}\right)v_\alpha.
\end{equation}
By Remark~\ref{rem-curv} and Proposition~\ref{lem:Tim-16}, there exists $C_*$ 
such that the quantity in parentheses
on the \textsc{rhs} of~\eqref{eq:v-evolution} is bounded in absolute value by
\begin{equation}
\label{eq-12345}
C_*\left(\sqrt{\frac{H_\alpha\big(v_\alpha(x,t)\big)}{v_\alpha(x,t)^2}}+
\frac{G_\alpha\big(v_\alpha(x,t)\big)}{v_\alpha(x,t)^2}\right)\leq C_* \| G_\alpha\|,
\end{equation}
where we have used that $|\nabla^2 v_{\alpha}| \le \chi_{\alpha}^{1/2} + C_*\, \sum_{\beta=1}^A\frac{\gamma_{\beta}^{1/2}}{u_{\beta}}\, \gamma_{\alpha}^{1/2}$. So by Proposition~\ref{lem:Tim-16}, we have
\begin{equation*}
\begin{split}
\frac{|\nabla^2 v_{\alpha}|}{v_{\alpha}}&\le \frac{\chi_{\alpha}^{1/2}}{v_{\alpha}} + C_*\, \left(\sum_{\beta=1}^A \frac{\gamma_{\beta}}{v_{\beta}^2}\right)^{1/2}\, \frac{\gamma_{\alpha}^{1/2}}{v_{\alpha}} \\
&\le C_* \, \sqrt{\frac{H_\alpha\big(v_\alpha(x,t)\big)}{v_\alpha(x,t)^2}}.
\end{split}
\end{equation*}
This yields \eqref{eq-12345}.
Hence $\partial_tv_\alpha\leq C_NC_* \| G_\alpha \|\,v_\alpha$. This proves the first claim.
\medskip

To prove the second claim, we observe that it follows from assumption~\eqref{eq:decay} at $t=0$
and the first claim that for any $\ve>0$, there exists $\delta>0$ such that if
$v_\alpha(x,0)<\delta$, then  $G_\alpha\big(v_\alpha(x,t)\big)<\ve\,v_\alpha^2(x,t)$ uniformly for $t\in[0,T)$.
At any $x$ with $v_\alpha(x,0)<\delta$, one can then bound the quantity in parentheses
on the \textsc{rhs} of~\eqref{eq:v-evolution} in absolute value by
\[
\sqrt{\frac{H_\alpha\big(v_\alpha(x,t)\big)}{v_\alpha(x,t)^2}}+
\frac{G_\alpha\big(v_\alpha(x,t)\big)}{v_\alpha(x,t)^2}\leq \ve.
\]
Hence at such $x$, one has $\partial_t v_\alpha(x,t)\leq \ve\,C_N\,v_\alpha(x,t)$.
The second claim follows.
\end{proof}

\subsection{Proofs of main results}	\label{sec:main-proof}

In this section, we prove Theorem~\ref{main} and Corollary~\ref{shrink}.

\begin{proof}[Proof of Theorem~\ref{main}]

We define
\[
T_{\sup}:=\sup\{T\in[0,T_{\sing})\colon
\mbox{ the conclusions of Theorem~\ref{main} hold on } [0,T]\}.
\]
We prove the theorem in two steps.
\smallskip

\emph{Step 1.} We claim that  $T_{\sup}>0$.
To see this, we recall that by \cite{Shi89} and \cite{CZ06}, there exists $T_{\min}(n,n_\alpha,C_\init)>0$
such that Ricci flow with initial data $(\mc M,g_\init)$ has a unique smooth solution on $[0,T_{\min}]$. 
We first apply Propositions~\ref{lem:Tim-15}
and~\ref{cor-improvement} with $t_0=0$, $t_1=T_{\min}$, $C_0=1$, and $C_1=C_1(T_{\min})<\infty$.
They yield a constant $C'$ and a time $T_1\in(0,T_{\min}]$ such that the estimates
\begin{align*}
v_\alpha(x,t) &\geq \frac{v_\alpha(x,0)}{1+C't},\\
\rho(x,t) &\leq C_{\init}\,(1+C't),\\
\gamma_\alpha(x,t) &\leq C_{\init} (1 + C't)\,G_\alpha\big(v_\alpha(x,t)\big),\\
\chi_\alpha(x,t) &\leq C_{\init} (1 + C't)\, H_\alpha\big(v_\alpha(x,t)\big),
\end{align*}
hold on $[0,T_1]$. We may then choose $T_2\in(0,T_1]$ small enough that $C'T_2\leq1$.
This ensures that the hypotheses of Propositions~\ref{lem:Tim-16} and~\ref{uniform} are satisfied on $[0,T_2]$.

Next we apply Propositions~\ref{lem:Tim-16} and~\ref{uniform}
on $[0,T_2]$. They yield a time $T_3\in(0,T_2]$ depending
only on $\{C_{\init},n,n_\alpha\}$ such that the estimates claimed in Theorem~\ref{main} hold on $[0,T_3]$.
It follows that $T_{\sup}\geq T_3>0$, thus proving the claim.

\emph{Step 2.} We next claim that $T_{\sup}\geq\min\{T_{\sing},C_*^{-1}\}$. We prove the claim by
contradiction, so we may
suppose that $T_{\sup}<\min\{T_{\sing},C_*^{-1}\}$. Then because $T_{\sup}<T_{\sing}$ and the inequalities
in Theorem~\ref{main} are of the form $\leq$ rather than $<$, they hold on $[0,T_{\sup}]$ by continuity. 
So we may apply Propositions~\ref{lem:Tim-15} and~\ref{cor-improvement} with $t_0=T_{\sup}$,
$t_1=T_1\in(T_{\sup},T_{\sing})$ arbitrary, $C_0=1+C_*T_{\sup}$, and $C_1=C_1(T_1)<\infty$.
They yield a constant $C''$ and a time $T_4\in(T_{\sup},T_1]$ such that the estimates
\begin{align*}
v_\alpha(x,t) &\geq \frac{v_\alpha(x,T_{\sup})}{1+C''(t-T_{\sup})}\geq \frac{v_\alpha(x,0)}{C_*\big(1+C''(t-T_{\sup})\big)},\\
\rho(x,t) &\leq (1+C_*T_{\sup})\,C_{\init}\,\big(1+C''(t-T_{\sup})\big),\\
\gamma_\alpha(x,t) &\leq (1+C_*T_{\sup})\,C_{\init} \big(1 + C''(t-T_{\sup})\big)\,G_\alpha\big(v_\alpha(x,t)\big),\\
\chi_\alpha(x,t) &\leq (1+C_*T_{\sup})\,C_{\init} \big(1 + C''(t-T_{\sup})\big)\, H_\alpha\big(v_\alpha(x,t)\big),
\end{align*}
hold on $[0,T_4]$. Because $1+C_*T_{\sup}<2$ by assumption, these estimates let us apply
Propositions~\ref{lem:Tim-16} and~\ref{uniform}
and thus obtain the conclusions of Theorem~\ref{main} on $[0,T_5]$ for some $T_5>T_{\sup}$.
By definition of $T_{\sup}$, this is a contradiction, which proves the claim.\end{proof}

\begin{proof}[Proof of Corollary~\ref{shrink}]
We recall that by assumption, at least one fiber is a positively curved space form, and  that we have chosen $\varsigma$ so that
$\frac{a_\varsigma}{\mu_\varsigma}=\min\{\frac{a_\alpha}{\mu_{\alpha}}:\mu_\alpha>0\}$.
Because the constant $C_*$ in Theorem~\ref{main} depends only on our Main Assumptions, which are independent of 
$a_\varsigma$, we may without creating circular dependencies shrink $a_\varsigma$ (where, abusing notation,
we continue to denote this quantity by $a_\varsigma$) to create new initial data $g'_{\init}$ for which $a_{\varsigma} =\mu_\varsigma/C_*$.
Note that we do not change $v_\alpha(\cdot,0)$.

Now we apply Theorem~\ref{main} to Ricci flow originating from $g'_{\init}$. The Theorem controls
the evolving metric on $[0,\min\{T_{\sing},C_*^{-1}\})$, for the same constant $C_*$. 
Since the $v_\alpha$ are uniformly equivalent in time on $[0,\min\{T_{\sing},C_*^{-1}\})$, the condition
that $\inf_{x\in\mc B}v_\alpha(x,t)=0$, which holds at $t=0$ by construction (see Remark~\ref{inf}), also holds
on that entire interval. But this means that $T_{\sing}$ can be no larger than the formal vanishing time
$T_{\form} = a_\varsigma/\mu_\varsigma=C_*^{-1}$. This in particular implies that the conclusion of Theorem~\ref{main} holds for all $t\in [0,T_{\sing})$. However, we have $T_{\sing} = T_{\form}$, because for $t\in [0,C_*^{-1})$,
Theorem \ref{main} implies positivity of $v_{\varsigma}$, hence that $u_{\varsigma} > a_{\varsigma} - \mu_{\varsigma} t > 0$,
and gives control on the remaining curvatures.
\smallskip

Next we prove that solutions originating from initial data satisfying our Main Assumptions develop Type-I singularities at spatial infinity.
We recall inequality~\eqref{eq:curv-near-cyl}:
\[
\Big|\Rm[g]-\sum_{\alpha=1}^A u_\alpha^{-1}\Rm[g_{\mc F_\alpha}]\Big|_g
	\leq C\left\{\rho^{1/2}+\sum_{\alpha=1}^A(u_\alpha^{-2}\gamma_\alpha+u_\alpha^{-1}\chi_\alpha^{1/2})\right\}.
\]
Theorem~\ref{main} implies that the \textsc{rhs} is bounded by a constant $C'$ depending only on
$C_*=C_*\big(C_{\init},\vec N=(N,n_\alpha)\big)$
and $T_{\sing}$. Thus we find that as $t \nearrow T_{\sing}$,
\begin{align}\label{eq:asymptotic-rm}
  |\Rm| = \sum_{\alpha = 1}^A \frac{c_\alpha}{u_\alpha} + C',
\end{align}
where the constants $c_\alpha$ depend only on $\vec N$ and the Ricci constants $\mu_\alpha$.
Moreover, Proposition~\ref{uniform} implies that on any compact set, the functions
\[
u_\alpha(x, t) = (a_\alpha - \mu_\alpha t) + v_\alpha(x, t)
\]
are bounded from below, and thus the curvature is bounded from above.
But by Remark~\ref{inf}, $\inf v_\alpha(x, 0) = 0$, and by Theorem~\ref{main}, this remains true for all $t>0$ that the solution exists.
Thus at any such time, the warping function of the fiber $\mc F_\varsigma$ satisfies
\[
\sup_{x\in\mc B} u_\varsigma^{-1}(x, t) = (a_\varsigma - \mu_\varsigma t)^{-1},
\]
which shows that the singularity is Type-I and forms at spatial infinity, as claimed.
\end{proof}

\section{Applications}
\label{sec-application}

\subsection{Essential blowup sequences on noncompact manifolds}
Assume that a Ricci flow solution $(\mc M,g(t))$ develops a singularity at some time $T<\infty$.
This means that $\limsup_{t\nearrow T}\mc R(t) = \infty$, where $\mc R(t) := \sup_{\mc M} |\Rm[g(t)]|$.
A Ricci flow solution $(\mc M,g(t))$ that becomes singular at $T<\infty$ is called \emph{Type-I}
if there exists a constant $C >0$ such that for all $t\in [0,T)$, one has $\mc R(t) \le \frac{C}{T-t}$.

If $\big(\mc M,g(t)\big)$ is a Type-I Ricci flow solution, then
a point $p\in\mc M$ is called a \emph{Type-I singular point} if there exists an essential blowup sequence
 $(p_i,t_i)\in\mc M\times[0,T)$ so that
$\lim_{i\to\infty} p_i = p$ and $\lim_{i\to\infty} t_i = T$. To be an \emph{essential blowup sequence} means that there
exists a constant $c > 0$ such that
\[|\Rm[g(t_i)]|_{g(t_i)} (p_i,t_i) \ge \frac{c}{T-t_i}.\]
Because $\frac{\mr d}{\mr dt}\mc R(t)\leq C_n\mc R^2(t)$,
the curvature of a developing singularity  always grows at least at a Type-I
rate, and so such sequences always exist. If $\mc M$ is noncompact, however, it might be the case that
an essential blowup sequence does not limit to any Type-I singular point in $\mc M$.

In \cite{EMT}, it is proven that if $(\mc M,g(t))$ is a Type-I Ricci flow on $[0,T)$, and if $p\in\mc M$ is a
Type-I singular point,
then for every sequence $\lambda_j\to\infty$, the corresponding 
rescaled Ricci flows $(\mc M, g_j(t), p)$, defined on $[-\lambda_j T, 0)$
by $g_j(t) := \lambda_j\, g(T + \lambda_j^{-1}t)$ subconverge to a normalized nontrivial gradient shrinking Ricci soliton in canonical form. This is a solution $(\mc N, g, f)$ that exists on a time interval $(-\infty,T]$ and satisfies
\[
\Rc+\cv^2 f = \frac{1}{2(T-t)}g\qquad\mbox{ and }\qquad 
\frac{\partial}{\partial t}f=|\cv f|^2.
\]

The result in \cite{EMT} applies in the case that  $(\mc M, g(t))$ is a Type-I Ricci flow on a compact manifold $\mc M$,
or if $\mc M$ is noncompact and $p\in\mc M$ is a Type-I singular point. On the other hand, if $\mc M$ is noncompact and $(\mc M,g(t))$ is a Type-I flow with a singularity forming at spatial infinity, then a Type-I singular point may not exist ---
see the example suggested in Remark~1.3 of~\cite{EMT}. In this case, the results of~\cite{EMT} do not preclude the
possibility  that the limit along some blow up sequences is a nontrivial gradient shrinking Ricci soliton, while along some other blow up
sequences it is not. Our goal in this section is to use the results we have proven here to produce an example
exhibiting this phenomenon.
\medskip

Assume that $(\mc M,g(t))$ is a complete noncompact Type-I Ricci flow that develops a singularity at spatial infinity at
some time $T<\infty$.
That is, there exist sequences $p_j\in\mc M$, $t_j\in [0,T)$, and a uniform constant $c_0 > 0$ such that
$(p_j,t_j) \to (\infty, T)$ as $j\to\infty$, and
\[|\Rm|(p_j,t_j) \ge \frac{c_0}{T-t_j}.\]
This in particular implies that $(p_j,t_j)$ is an essential blow up sequence for the singularity developing at
spatial infinity at time $T$. If $\lambda_j$ is a sequence such that $\lim_{j\to\infty} \lambda_j = \infty$
and if $g_j(\cdot,t) := \lambda_j\, g(\cdot, T + t \lambda_j^{-1})$, then
a \emph{blowup limit} of the flow along the sequence $p_j$ is a pointed subsequential limit of
$(\mc M, g_j(\cdot,T), p_j)$, if it exists.
\smallskip

We now explore what blowup limits are possible for a particular family of noncompact Ricci flow
solutions of the type considered in this paper.

\begin{definition}
\label{def-adm-pert}
Let $g_{\mr{Eucl}}$ denote the Euclidean metric on $\mb R^k$.
A family of functions $\delta_{\alpha}(x) : \mathbb{R}^k \to \mb R_+$  specifies an
\emph{admissible perturbation} of
$g = g_{\mr{Eucl}} + \sum_{\alpha=1}^A a_\alpha\,g_{\mathcal{F}_\alpha^{n_\alpha}}$
on $\mathbb{R}^k\times \mathcal{F}^{n_1}\times\cdots\times\mathcal{F}^{n_A}$
if $\lim_{|x|\to\infty} \delta_{\alpha}(x) = 0$ and there exist functions $G_{\alpha} \in \mathcal{G}$ satisfying
\[
\frac{G_{\alpha}(\delta_{\alpha}(x))}{\delta_{\alpha}^2(x)} = o(1; \delta_{\alpha}(x) \searrow0)
\]
such that the metric
\begin{equation}	\label{eq:A-metric-ansatz}
g_{\init} = g_{\mr{Eucl}} + \sum_{\alpha=1}^A\big(a_{\alpha} + \delta_{\alpha}(x)\big)\, g_{\mathcal{F}^{n_{\alpha}}}
\end{equation}
satisfies the Main Assumptions.\footnote{We note that by Remark~\ref{positive}, satisfying the Main Assumptions
forces the functions $\delta_{\alpha}$ to be strictly positive everywhere.}
\end{definition}

We define the set
\begin{equation}
\label{eq-A}
\mathcal{A} := \{\delta_\alpha(x): \mathbb{R}^k\to\mb R_+\colon\delta_\alpha(x)
\mbox{ is admissible in the sense of Definition~\ref{def-adm-pert}}\}.
\end{equation}

\begin{lemma}
For any manifold $\mb R^k\times\mc S^p\times\mc S^q$, the set $\mc A$ is nonempty.
\end{lemma} 

\begin{proof}
In polar coordinates on $\mb R^k$, choose rotationally-symmetric warping functions
$\delta_1(r) =\delta_2(r) = \frac{1}{1+r^2}$, along with control functions
$G_1(s) =G_2(s) = \frac{s^3}{1+s}$.
Then it is straightforward to verify that the Main Assumptions are satisfied for the metric~\eqref{eq:A-metric-ansatz}, because
\[
\gamma_\alpha(r)=
|\nabla\delta_\alpha(r)|^2 = \frac{4r^2}{(1+r^2)^4}
\le \frac{4}{(1+r^2)^3} \le  8\,G_\alpha(\delta_\alpha(r)),\]
and 
\begin{align*}
  \chi_\alpha(r) = |\nabla \nabla \delta_\alpha(r)|^2
  &\le C_{init}\left(\frac{1}{(1+r^2)^4} + \frac{r^4}{(1+r^2)^6}\right) \\
  &\le \frac{C_{\init}}{(1+r^2)^4} \le C_{\init} H_\alpha(\delta(r)).
\end{align*}
Here we use the fact that $H_{\alpha}(\delta_{\alpha}(r)) \ge \frac{1}{2(1+r^2)^4}$, which is easy to check.
We also have $\rho\equiv0$ for the Euclidean metric. Thus we conclude that
$\delta_1,\delta_2 \in \mathcal{A}$.
\end{proof}

For the purpose of the applications that we discuss in this subsection, it suffices to consider a doubly-warped product.
Thus we fix $k=1$, $A=2$, and spherical fibers $\mc{F}^{n_1} =\mc F^{n_2}= \mc S^p$ $(p\geq2)$
in the remainder of this subsection.

To prove Theorem~\ref{thm-spatial-infinity}, we consider $\mc M = \mb R \times \mc S^p\times \mc S^p$
$(p\geq2)$, with initial metric 
\begin{equation}	\label{eq:A-initial-data}
g_{\init} = (\mr dx)^2 + \big(a_* + v_1(x,0)\big)\, g_{\mc S^p} + \big(a_* + v_2(x,0)\big)\, g_{\mc S^p},
\end{equation}
where $g_{\mc S^p}$ is the round metric scaled so that $2\Rc_{g_{S^p}} = g_{\mc S^p}$,
and  $v_1(x,0) = \delta_1(x)$ and $v_2(x,0) = \delta_2(x)$, where $\delta_1,\delta_2\in\mc A$
are functions such that
\[
\lim_{|x|\to\infty}\frac{\delta_1(x)}{\delta_2(x)} = \eta \in \mb R_+\setminus\{0,1\}.
\]
We require $\eta>0$ to ensure that $\delta_1$ and $\delta_2$ remain comparable, so that we may take
appropriate limits. We further require that $\eta\neq1$ to demonstrate the existence of sequences that
cannot limit to nontrivial gradient shrinking Ricci solitons.
\smallskip

A slight modification of the construction in this section shows that there exist noncompact Ricci flow solutions that develop Type-I
singularities for which there can be no blowup limits $(\mc M_\infty, g_\infty)$. Indeed,
if we consider $\mc M=\mb R\times\mc S^1\times\mc S^p$ with an initial metric that is not $\kappa$-noncollapsed, 
\[
g_{\init}=(\mr dx)^2+\delta_1(x)(\mr d\theta)^2+(a_*+\delta_2(x))g_{\mc S^p},
\]
where $\delta_1,\delta_2\in\mc A$, then our work in proving Theorem~\ref{thm-spatial-infinity} goes through, \emph{mutatis mutandis}, and establishes the following:

\begin{remark}
There exist complete, noncompact, collapsed Ricci flow solutions $\big(\mc M,g(t)\big)$ that develop Type-I
singularities at spatial infinity. For each of these solutions, Type-I blowups have no Cheeger--Gromov limits.
\end{remark}
As we note in the introduction, blowup limits may exist as \'etale groupoids, in the
sense considered in~\cite{Lott10}.
\smallskip

\begin{proof} [Proof of Theorem~\ref{thm-spatial-infinity}]
By the proof of Corollary~\ref{shrink},  we may choose $a_* > 0$ sufficiently small so that the estimates
of Theorem~\ref{main} hold for the solution originating from initial data~\eqref{eq:A-initial-data}
up to the singular time $T_{\sing} = a_*$, at which time it encounters a Type-I singularity.
For as long as it exists, the Ricci flow solution has the form 
\begin{equation}
\label{eq-as-long-as}
g(x,t) = \mr{d}x^2 + \big(a_* - t + v_1(x,t)\big)\,g_{\mc S^p} + \big(a_* - t + v_2(x,t)\big)\,g_{\mc S^p}.
\end{equation}
By the second part of Proposition~\ref{uniform}, we have
\begin{equation}
\label{eq-unif-equiv-small}
v_\alpha(x,t)=\big(1+o(1;\,\delta_\alpha(x)\searrow0)\big)\,\delta_\alpha(x)
\end{equation}
for all $t\in [0,a_*)$.

Initially, we consider \underline{any} sequence $(x_j,t_j) \to (\infty,a_*)$ for which
\begin{equation}	\label{eq:define-c-alpha}
\lim_{j\to\infty} \frac{\delta_{\alpha}(x_j)}{a_*-t_j} =: c_{\alpha} \in (0,\infty),\qquad (\alpha=1,2).
\end{equation}
Then our curvature estimates above easily imply that there exists a uniform constant $c_0 > 0$ such that
$|\Rm(x_j,t_j)| \ge \frac{c_0}{a_* - t_j}$.
Let $\lambda_j \to \infty$ be any sequence such that $\limsup_{j\to\infty} (T-t_j)\, \lambda_j < \infty$, and
let $g_j(t) = \lambda_j \, g(a_*+t\lambda_j^{-1})$. This rescaling ensures that each $g_j$ exists up to $t=0$;
indeed, because we have shown that the singularity is Type-I, it immediately follows that 
\begin{equation}
\label{eq-curv-bound-j}
|\Rm[g_j(t)]| \le \frac{C}{-t}
\end{equation}
for a uniform constant $C$. Note that in our discussion of convergence below, we always mean in the sense of
subsequential  convergence, even if we do not explicitly pass to subsequences.
\medskip

We rewrite $g_j(t) = dy^2 + u_{1j}\, g_{S^p} + u_{2j} \, g_{S^p}$,
where 
\begin{equation*}
\begin{split}
u_{\alpha j}(y,t) &= \lambda_j\, u_{\alpha}(x_j + y\lambda_j^{-1}, a_* + t \lambda_j^{-1}) \\
&= (-t) + \lambda_j v_{\alpha}(x_j + y\lambda_j^{-1}, a_* + t\lambda_j^{-1}).
\end{split}
\end{equation*}
Because our manifold is $\mc M = \mb R \times \mc S^p\times \mc S^p$ with $p\geq2$, the metrics $g_j$
are $\kappa$-noncollapsed. Hence by \eqref{eq-curv-bound-j}, a pointed sequence of Ricci flow solutions
$(\mc M, g_j(t), x_j)$ smoothly converges in the 
Cheeger--Gromov sense to an ancient Ricci flow solution $(\mc M_{\infty}, g_{\infty}(t), o)$ that exists for $t\in (-\infty, 0)$.

\begin{claim}	\label{claim-structure}
There exist smooth limits $u_{\alpha\infty}(y,t)$ of $u_{\alpha j} (y,t)$ as $j\to\infty$.

Thus the limit of the convergent subsequence $(\mc M, g_j(t), x_j)$ is
$\mc M_{\infty} = \mathbb{R}\times \mc S^p\times \mc S^p$ with the metric
\[
g_{\infty}(y,t) = dy^2 + u_{1\infty}(y,t) g_{\mc S^p} + u_{2\infty}(y,t)\, g_{\mc S^p}.
\]
\end{claim}

\begin{proof}
Recall that by Theorem \ref{main}, we have
\[c_*\,\delta_{\alpha}(x_j + y\lambda_j^{-1}, 0) \le  v_{\alpha}(x_j + y\lambda_j^{-1},a_* + t\lambda_j^{-1}) \le C_* \delta_{\alpha}(x_j+y\lambda_j^{-1},0)\]
for uniform constants $c_*, C_*$. Putting $y = 0=o$ as above, this yields
\[c_* \lambda_j \delta_{\alpha}(x_j,0) \le v_{\alpha j}(0,t) \le C_* \lambda_j \delta_{\alpha}(x_j,0).\]
Because $\lim_{j\to\infty} \frac{\delta_{\alpha}(x_j)}{a_* - t_j} = c_{\alpha} \in (0,\infty)$ and
$0 < \limsup_{j\to\infty} \lambda_j (T- t_j) < \infty$, we immediately get
\begin{equation}
\label{eq-uniform-bd-point}
c_0 \le v_{\alpha j}(0,t) \le C_0
\end{equation}
for all $t\in (-a_* \lambda_j, 0)$, for uniform constants $0 < c_0 < C_0 < \infty$.

On the other hand, by Theorem \ref{main}, we also have
\[|\nabla v_{\alpha}(x,t)|^2 \le C_0 \, G\big(v_{\alpha}(x,t)\big)\qquad \mbox{for all} \qquad t\in [0,a_*),\]
where $C_0$ is another uniform constant. This in particular implies that
\begin{equation}
\label{eq-grad-bd-conv}
|\nabla \log u_\alpha(x, t)|^2 \leq |\nabla \log v_{\alpha}(x,t)|^2 \le C_0 \qquad \mbox{on} \qquad\mc M \times [0,a_*).
\end{equation}
Estimates \eqref{eq-uniform-bd-point} and \eqref{eq-grad-bd-conv} imply that $v_{\alpha j}(y,t)$ converges uniformly 
to $v_{\alpha\infty}(y,t)$ on compact sets of $\mc M\times [-a_*\lambda_j, 0)$ as $j\to\infty$, in a $C^{0,\mu}$ norm,
for some $\mu\in (0,1)$. This together with a smooth Cheeger--Gromov convergence implies the claim. 
\end{proof}

\begin{claim}
\label{claim-constant-soliton}
For $\alpha\in\{1,2\}$ and every $t\in (-\infty,0)$, both $u_{\alpha\infty}(y,t)$ are constant in space.
\end{claim}

\begin{proof}
We fix any $t\in(-\infty,0)$, and let $t_j=a_*+t\lambda_j^{-1}$. 
 Then we observe that estimate~\eqref{eq-grad-bd-conv} scales as follows:
  \begin{align*}
    \lambda_j^2 |\nabla_{g_j(t_j)} \log u_{\alpha j}|^2 \leq C_0.
  \end{align*}
Taking $j \to \infty$ and using the smooth convergence of the metrics proves that $\log u_{\alpha \infty}$ is constant in space.
\end{proof}

To finish the proof of the first part of Theorem \ref{thm-spatial-infinity},
we need to show that the limit $(\mc M, g_{\infty}(\cdot,t),o)$ cannot be a gradient shrinking Ricci soliton
if
\[
\lim_{|x|\rightarrow\infty}\frac{\delta_1(x)}{\delta_2(x)}=\eta\notin\{0,1\}
\]
\underline{and} if the spacetime sequence $(x_j,t_j)$ is such that the constants $\lim_{j\to\infty}\frac{\delta_{\alpha}(x_j)}{a_*- t_j} = c_{\alpha}$
defined in~\eqref{eq:define-c-alpha} satisfy $c_1 = c_2 \eta$ with $\eta\in \mathbb{R}_+\backslash\{0,1\}$.
Using Proposition~\ref{uniform} and~\eqref{eq-as-long-as}, we have
\begin{equation}
\label{eq-ratio-u1u2}
\begin{split}
\frac{u_{1j}(0,-(a_*-t_j)\lambda_j)}{u_{2j}(0,-(a_*-t_j)\lambda_j)} &= \frac{u_1(x_j,t_j)}{u_2(x_j,t_j)} \\
&= \frac{a_*-t_j + (1 + o(1, \delta_1(x_j)\searrow 0)\, \delta_1(x_j)}{a_*-t_j + (1 + o(1, \delta_2(x_j)\searrow 0)\, \delta_2(x_j)} \\
&= \frac{1 + (1 + o\Big(1, \delta_1(x_j)\searrow 0)\Big)\, \frac{\delta_1(x_j)}{a_*-t_j}}{1 + (1 + o\Big(1, \delta_2(x_j)\searrow 0)\Big)\, \frac{\delta_2(x_j)}{a_*-t_j}}.
\end{split}
\end{equation}
Recall that  $\lim_{j\to\infty} (a_* - t_j) \lambda_j = -t_0 \in (0,\infty)$.   We let $j\to\infty$ in \eqref{eq-ratio-u1u2} to obtain
\[\frac{u_{1\infty}(0,t_0)}{u_{2\infty}(0,t_0)} = \frac{1 + c_2\eta}{1 + c_2} \neq 1.\]
In particular, by Claim \ref{claim-constant-soliton}, we have
\begin{equation}
\label{eq-ratio-not-one}
\frac{u_{1\infty}(y,t_0)}{u_{2\infty}(y,t_0)} = \frac{1 + c_2\eta}{1 + c_2} \neq 1 \qquad \mbox{for all} \qquad y\in \mathbb{R}.
\end{equation}

\begin{claim}
At no time $t\in (-\infty,0)$ is $g_{\infty}(\cdot,t)$ a gradient shrinking Ricci soliton.
\end{claim}

\begin{proof}
Recall that a gradient shrinking soliton is a metric $g$ that satisfies
\[
-2\Rc+\mc L_X(g) = \lambda g,
\]
where $X$ is the gradient vector field of a potential function and  $\lambda<0$.
It is shown in~\cite{AK19} that a metric
\[
g=(\mr dy)^2 + \vp_1(y)^2g_{\mc S^{p_1}} + \vp_2(y)^2g_{\mc s^{p_2}}
\]
on a doubly-warped product $\mb R\times\mc S^{p_1}\times\mc S^{p_2}$
is a gradient shrinking soliton with vector field $X=f(y)\,\frac{\partial}{\partial y}$ if and only if
the functions $f,\vp_1,\vp_2$ satisfy the \textsc{ode} system
\begin{subequations}	\label{eq-ODE}
\begin{align}
f_y &= p_1\frac{(\vp_1)_{yy}}{\vp_1}+p_2\frac{(\vp_2)_{yy}}{\vp_2}-\lambda,\\
\frac{(\vp_1)_{yy}}{\vp_1} &= (p_1-1)\frac{1-(\vp_1)_y^2}{\vp_1^2} -
p_2\frac{(\vp_1)_y(\vp_2)_y}{\vp_1\vp_2}
+ \frac{(\vp_1)_y}{\vp_1}f+\lambda,\\
\frac{(\vp_2)_{yy}}{\vp_2} &=(p_2-1)\frac{1-(\vp_2)_y^2}{\vp_2^2} -
p_1\frac{(\vp_1)_y(\vp_2)_y}{\vp_1\vp_2} + \frac{(\vp_2)_y}{\vp_2}f +\lambda.
\end{align}
\end{subequations}
The only solutions of this system with $\vp_1$ and $\vp_2$ constant in space are
\[
f(y)=-\lambda y,\qquad
\vp_1^2=\frac{p_1-1}{-\lambda},\quad\mbox{ and }\quad \vp_2^2=\frac{p_2-1}{-\lambda}.
\]
So if $g_{\infty} = dy^2 + u_{1\infty} g_{S^p} + u_{2\infty} g_{S^p}$ were a gradient shrinking Ricci soliton at some 
$t\in(\-\infty,0)$, then the constants $u_{1\infty}$ and $u_{2\infty}$ would have to be equal. But this
contradicts~\eqref{eq-ratio-not-one}.
\end{proof}
\medskip

We now prove the remainder of Theorem \ref{thm-spatial-infinity} and Corollary~\ref{necessary-and-sufficient}
by obtaining necessary and sufficient conditions for a limit to be a gradient shrinking soliton.

If $\lambda_j \to \infty$ is such that  $\lim_{j\to\infty} \lambda_j (a_* - t_j) = -t_0 > 0$, 
estimate~\eqref{eq-ratio-u1u2} implies that
\begin{equation}	\label{eq-ratio-one}
\frac{u_{1\infty}(0,t_0)}{u_{2\infty}(0,t_0)} = 1
\end{equation}
if and only if $(x_j,t_j)$ is a sequence converging to $(\infty, a_*)$ such that
\begin{equation}	\label{eq:zero-limit}
\lim_{j\to\infty}\frac{\delta_{\alpha}(x_j)}{a_* - t_j} = 0\qquad\mbox{ for }\quad \alpha\in \{1,2\},
\end{equation}
in contrast to~\eqref{eq:define-c-alpha}.

If~\eqref{eq:zero-limit} holds, then Claim~\ref {claim-constant-soliton} implies that for every $t\in (-\infty,0)$,
we have $u_{1\infty}(y,t) = u_{2\infty}(y,t) = u(t)$, where $u$ depends only on time.
Thus $u_{1\infty} = u_{2\infty} = -\lambda (p -1)$ and $f(y) = -\lambda y$ satisfy the system~\eqref{eq-ODE},
implying that the metric
$g_{\infty} = dy^2 + u_{1\infty}\, g_{S^p} + u_{2\infty}\, g_{S^p}$ is a gradient shrinking Ricci soliton. 
\smallskip

Finally, since we have the bounds $\rho + \sum_{\alpha = 1}^Au_\alpha^{-1}\chi_\alpha + u^{-2}_\alpha\gamma_\alpha \leq C_0$ for all $t\in [0,a_*)$, the curvature estimate~\eqref{eq:curv-near-cyl} implies that
\begin{equation}
\label{eq-Rm-sup}
\sup_{\mc M} |\Rm[g(t)]| = \sup_{\mc M} \frac{\mu_{\alpha}}{2(p-1) u_{\alpha}} = \frac{\mu_{\alpha}}{2(p-1)(a_*-t)}.
\end{equation}
Note that to obtain the last identity, we use
Proposition~\ref{uniform} and the fact that $\lim_{|x|\to\infty}\delta_{\alpha}(x) = 0$.
On the other hand, Proposition \ref{uniform} also implies that
\[v_{\alpha}(x_j,t_j) = \Big(1 + o(1; \delta_{\alpha}(x_j)\searrow 0)\Big)\, \delta_{\alpha}(x_j).\]
Combining this with \eqref{eq:curv-near-cyl}, we find that
\begin{equation}
\label{eq-Rm-j}
|\Rm(x_j,t_j)| =
\frac{\left( 1 + o(1) \right)}{a_* - t_j}
\frac{\mu_{\alpha}}{2(p-1)\Big(1 + (1 + o(1; \delta_{\alpha}(x_j)\searrow 0) \frac{\delta_{\alpha}(x_j)}{a_*-t_j}\Big)} + O(1).
\end{equation}
Thus \eqref{eq-Rm-sup} and \eqref{eq-Rm-j} imply that
\[
\lim_{j\to\infty}\frac{|\Rm(x_j,t_j)|}{\sup_{\mc M} |\Rm(\cdot,t_j)|} = 1
\]
if and only if~\eqref{eq:zero-limit} holds. We have seen above that~\eqref{eq:zero-limit} is equivalent to~\eqref{eq-ratio-one},
and that by Claim~\ref{claim-constant-soliton}, \eqref{eq-ratio-one} is equivalent to $g_{\infty}$ being a gradient shrinking Ricci soliton.
This concludes the proof of Theorem~\ref{thm-spatial-infinity} and verifies Corollary~\ref{necessary-and-sufficient}.
\end{proof}

\subsection{Weak stability of generalized cylinders under Ricci flow}

Stability of generalized cylinders $\mathbb{R}^k\times \mc S^p$ under Ricci flow is a subtle question. Even though a round cylinder $\mathbb{R}^k\times \mc S^p$ is expected to be a stable singularity model in some sense, it is not immediately clear how to define this stability.
One reason for this is the following example. Start with a cylindrical metric
$g_{\cyl} = (\mr dx)^2 + g_{\mc S^p}$ on $\mathbb{R}\times \mc S^p$ with $p\geq2$.
Let $T_0$ denote the time at which the spherical fibers vanish. Now consider an $\epsilon$-perturbation of
the initial data $g_{\cyl}$:
 an initial metric $g_{\epsilon} = (\mr dx)^2 + (1+\epsilon)\, g_{\mc S^p}$ with $|\epsilon|<1$.
Ricci flow originating from $g_{\epsilon}$ will also become singular but at a different singularity time.
If we rescale the perturbed flow by $\frac{1}{T_0 - t}$, then the rescaled perturbed solution will encounter a
singularity before $T_0$ if $\epsilon<0$ or will become infinitely large as $t\nearrow T_0$ if $\epsilon>0$.
In other words, no matter how small a perturbation is, if we chose a cylinder of a different radius, it will not naturally
converge after rescaling to the solution originating at $g_{\cyl}$.
\medskip

Now let $g_{\mr{Eucl}}$ denote the flat Euclidean metric on $\mb R^k$, and let $\mc S^p$ be a round sphere scaled so that 
$2\Rc_{g_{\mc S^p}} = g_{\mc S^p}$. We take as initial data
$g_{\cyl}(0) = g_{\mr{Eucl}} + a_* g_{\mc S^p}$.
Then $g_{\cyl}(t) = g_{\mr{Eucl}} + (a_*-t)\, g_{\mc S^p}$
is a generalized cylinder that solves Ricci flow up to time $a_*>0$.
Consider perturbed initial data
\begin{equation}	\label{eq:perturbed-cylinder}
g(x,0) = g_{\mr{Eucl}} + u(x,0) g_{\mc S^p},
\end{equation}
where $u(x,0) = a_* + \delta(x)$, with $\delta(x) \in \mathcal{A}$ as defined in~\eqref{eq-A}.

\begin{theorem}
\label{cor-stability}
Let $g(x,t)$ be a Ricci flow solution on $\mb R^k\times S^p$ with initial metric $g(x,0)$ given in~\eqref{eq:perturbed-cylinder}.
Then there exists a constant $C_*$ depending only on $a_*$ so that for all $x\in\mb R^k$ and all $t\in [0,a_*)$, one has
\begin{equation}
\label{eq-concl1}
\begin{split}
\frac{1}{C_*}\, \delta(x) &= \frac{1}{C_*}\, |g(x,0) - g_{\cyl}(0)|_{g_{\cyl}(0)}\\
&\le \sup_{t\in[0,a_*)}|g(x,t) - g_{\cyl}(t)|_{g_{\cyl}(0)} \\
&\le C_* |g(x,0) - g_{\cyl}(0)|_{g_{\cyl}(0)} = C_* \, \delta(x),
\end{split}
\end{equation}
and
\begin{equation}
\label{eq-concl2}
\begin{split}
\sup_{t\in[0,a_*)} |g(x,t) - g_{\cyl}(t)|_{g_{\cyl}(0)}
&= \Big(1 + o(1; \delta(x)\searrow 0)\Big)\,|g(x,0) - g_{\cyl}(0)|_{g_{\cyl}(0)} \\
&= \Big(1 + o(1; \delta(x)\searrow 0)\Big)\,\delta(x),
\end{split}
\end{equation}
which implies that the $g_{\cyl}(0)$-distance between a perturbed solution $g(x,t)$
and an evolving generalized cylinder $g_{\cyl}(t)$ approaches zero as $|x|\to\infty$, uniformly in time $t\in [0,a_*)$. 
Moreover, the flow $g(x,t)$ develops a Type-I singularity at spatial infinity as $t \nearrow a_*$. 
\end{theorem}  

\begin{remark}
We say that a solution $g(x,t)$ on $\mathbb{R}^k\times \mc S^p$ stays in a $\delta(x)$-neighborhood of $g_{\cyl}$
if there exists a uniform constant $C_*$ so that $\sup_{t\in[0,a_*)} |g(x,t) - g_{\cyl}(t)|_{g_{\cyl}(0)}$
is bounded by $C_*$ for all $x\in \mathbb{R}$ and all $t\in [0,a_*)$.
Note that for the admissible perturbations that we consider in Theorem~\ref{cor-stability},
the constant $C_*$ is universal: independent of the perturbation. Note also that Theorem~\ref{cor-stability}
implies that for admissible perturbations, the perturbed solution never leaves the $\delta(x)$-neighborhood of
$g_{\cyl}$.

On the other hand, \eqref{eq-concl1} implies that no matter how small $\delta(x) > 0$ may be,
after performing a Type-I rescaling by $\frac{1}{a_* - t}$ of both flows $g(x,t)$ and $g_{\cyl}(t)$, the rescaled solutions
$\tilde{g}$ and $\tilde{g}_{\cyl}$, respectively, have the property that $\tilde{g}(\cdot,\tau)$ does not converge to
$\tilde{g}_{\cyl}(\tau)$ as $\tau \to \infty$, where $\tau = -\log(a_* -t)$. This behavior is consistent with
the example discussed in the opening paragraph of this subsection.
\end{remark}

\begin{proof}[Proof of Theorem~\ref{cor-stability}]
By Proposition \ref{uniform}, if $g(x,t) = a_* - t + v(x,t)$, then there exists a constant $C_*(a_*)$ such that
$v(x,t) \le C_* \, \delta(x)$ and
\[v(x,t) = \Big(1 + o(1; \delta(x)\searrow 0)\Big)\, \delta(x)\]
for all $t\in [0,a_*)$. This implies \eqref{eq-concl1} and \eqref{eq-concl2}.
This further implies that the distance between the perturbed solution $g(x,t)$ and an evolving cylinder
$g_{\cyl}(t)$ approaches zero as $|x|\to\infty$, uniformly in time $t\in [0,a_*)$.

Arguments exactly like those that prove Theorem~\ref{thm-spatial-infinity}
establish that the perturbed solution $g(x,t)$ has the same singular time $a_*$ as the generalized
cylinder $g_{\cyl}(t)$; the singularity is Type-I; and it occurs at spatial infinity. 
\end{proof}

\appendix
\section{Curvatures of multiply-warped products}	\label{multiply-warped}

We begin by recalling classical formulas for the curvatures\footnote{Throughout this paper, we follow the
curvature conventions detailed in Sections~5--6 of~\cite{CK04}. Briefly, $R(X,Y)Z=\cv^2Z(X,Y)-\cv^2Z(Y,X)$
for the $(3,1)$-tensor, and we lower the raised index into the fourth position so that, say,
$R_{1221}>0$ on the round $2$-sphere.} of a simple  warped product $\mc B\times_u \mc F$.
Let $(\mc B, \check g)$ and $(\mc F,\hat g)$ be complete Riemannian manifolds. In this Appendix, unlike the rest of
this paper, we do not assume that $\mc F$ is a space form. Let $u\colon\mc B\rightarrow\mb R_+$ be
a smooth function. To facilitate working in local coordinates, we denote the metric on  $\mc B\times_u \mc F$
by $g=\check g + u \hat g$.
\smallskip

We begin by working in local coordinates, using lowercase Roman indices (\emph{e.g.,} $i,j,k,\ell$) on the base $\mc B$,
lowercase Greek indices (\emph{e.g.,} $\sigma,\tau,\nu,\omega$) on the fiber $\mc F$, and allowing capital Roman letters to range
over both sets. We denote the Christoffel symbols of $g$ by
\[
\Gamma_{IJ}^K= \frac12 g^{KL}(\pd_I g_{JL}+\pd_J g_{IL}-\pd_L g_{IJ}),
\]
and those of $\check g$ and $\hat g$ by $\check\Gamma_{ij}^k$ and $\hat\Gamma_{\sigma\tau}^\nu$,
respectively. We follow the same convention for other geometric quantities, including curvatures.
We order the Christoffel symbols by the number of vertical (Greek) indices that appear (in order: $0,1,2,3$) and calculate that
\begin{subequations}		\label{eq:Christoffel}
\begin{align}
\Gamma_{ij}^k&=\check\Gamma_{ij}^k,\label{eq:Christoffel:Base}\\ \notag \\
\Gamma_{\sigma j}^k &= \Gamma_{i\tau}^k = \Gamma_{ij}^\nu=0,\\ \notag \\
\Gamma_{\sigma\tau}^k&=-\frac12\,\check g^{k\ell}u^{-1}\,\pd_\ell u\,\left( u\hat g_{\sigma\tau}\right),\\
\Gamma_{i\tau}^\nu&=\frac12\,u^{-1}\,\pd_i u\,\delta_\tau^\nu,\\
\Gamma_{\sigma j}^\nu&=\frac12\,u^{-1}\,\pd_j u\,\delta_\sigma^\nu,\label{eq:Christoffel:BaseBottomRight}\\ \notag \\
\Gamma_{\sigma\tau}^\nu&=\hat\Gamma_{\sigma\tau}^\nu.
\end{align}
\end{subequations}

Given a function $f: \mc B \to \mb R$, there is a natural function $\tilde f:\mc B\times_u\mc F\to\mb R$
defined by $\tilde f(x,y)=f(x)$. We wish to compare the covariant Hessian of $\tilde f$ with respect to $g$
with that of $f$ with respect to $\check g$.

\begin{claim}
  \label{claim:hessf}
 If $f$ and $\tilde f$ are as above, then
\[
\nabla \nabla \tilde f = \check \nabla \nabla f + \big\lp\nabla (\log u^{1/2}),\cv f\big\rp\,(u g_{\mc F}).
\]
\end{claim}
\begin{proof}
  This is a straightforward application of~\eqref{eq:Christoffel}. We can write
  \begin{align*}
  \cv_I\cv_J \tilde f &= \check\cv_I\cv_J f + (\cv-\check\cv)_I\cv_J f\\
  &=\cv_i\cv_j f + g_{JK}(\cv-\check\cv)_{IL}^K\cv^L f.
  \end{align*}
  Since $\cv f$, is horizontal, the only quantity from~\eqref{eq:Christoffel} that appears in the last term above
  is~\eqref{eq:Christoffel:BaseBottomRight}, which proves the claim.
\end{proof}

We now compute the curvatures of $g$ at the origin of a coordinate system that is normal for $\check g$ and
$\hat g$, but not necessarily so for $g$. That is to say, we may assume that $\check\Gamma_{ij}^k=0$
and $\hat\Gamma_{\sigma\tau}^\nu=0$ at the origin, hence that $\pd_i \check g_{jk}=0$ and
$\pd_\sigma \hat g_{\tau\nu}=0$ there, but we must use the full formula
\[
R_{IJKL}=g_{LP}\left(\pd_I\Gamma_{JK}^P-\pd_J\Gamma_{IK}^P+\Gamma_{IQ}^P\Gamma_{JK}^Q-\Gamma_{JQ}^P\Gamma_{IK}^Q\right)
\]
to calculate the $(4,0)$-Riemann curvature tensor of $g$.  Again ordering formulas by the number of vertical indices that
appear (in order: $0,1,2,3,4$), we compute that
\begin{align*}
R_{ijk\ell}&=\check R_{ijk\ell},\\ \\ 
R_{\sigma jk\ell}&=R_{i\tau k\ell} = R_{ij\nu\ell}=R_{ijk\omega}=0,\\ \\
R_{\sigma\tau k\ell}&=0,\\
R_{i\tau\nu\ell}&=g_{p\ell}\big(\pd_i\Gamma_{\tau\nu}^p-\Gamma_{\tau\omega}^p\Gamma_{i\nu}^\omega\big)\\
&= (u \hat g_{\tau\nu})\Big(-\frac12u^{-1}\,\cv_i\cv_\ell u+\frac14\,u^{-2}\,\cv_i u\,\cv_\ell u\Big),\\\\
R_{\sigma\tau\nu\ell}&=0,\\ \\
R_{\sigma\tau\nu\omega}&=g_{\omega\lambda}\big(\hat R_{\sigma\tau\nu}^\lambda
+\Gamma_{\sigma m}^\lambda\Gamma_{\tau\nu}^m-\Gamma_{\tau m}^\lambda\Gamma_{\sigma\nu}^m\big)\\
&= u\hat R_{\sigma\tau\nu\omega}-\frac14u^{-2}|\cv u|^2
\big(\left(u \hat g_{\sigma\omega}\right)\left( u \hat g_{\tau \nu}\right)-\left(u \hat g_{\tau\omega}\right)\left(u \hat g_{\sigma\nu}\right)\big).
\end{align*}
For use below, we note that the curvature operator vanishes if a horizontal plane is paired
with a plane spanned by two vertical vectors, as follows easily from the observations
\begin{equation}	\label{eq:null-plane}
0=R_{\sigma\tau k}^\ell\,g_{j\ell}=R_{\sigma\tau kj}=R_{kj\sigma\tau}\qquad\mbox{ and }\qquad
0=R_{kj\sigma\tau}\,g^{\tau\nu}=R_{kj\sigma}^\nu.
\end{equation}
\smallskip

There is a more concise way to write these formulas. Recall that the Kulkarni--Nomizu product of symmetric
$(2,0)$-tensors $\Phi,\Psi$ is given by
\begin{equation}	\label{eq:KN-convention}
(\Phi\KN\Psi)_{IJKL} := \Phi_{IL}\Psi_{JK}+\Phi_{JK}\Psi_{IL}-\Phi_{IK}\Psi_{JL}-\Phi_{JL}\Psi_{IK}.
\end{equation}
With this normalization, the $(4,0)$-curvature tensor $\Rm$ of a metric $g$ of constant
sectional curvature $\kappa$ is given by $\Rm=\tfrac12\kappa\,g\KN g$. Noting that
\[
u^{-1/2}\check\cv\cv(u^{1/2})=\frac12 u^{-1}\check\cv\cv u -\frac14 u^{-2}\cv u \otimes\cv u
\]
and using the identity $u^{-2}|\cv u|^2=4|\cv(\log u^{1/2})|^2$,
one sees that the curvature formulas above are equivalent to
\begin{equation}	\label{eq:rm_single}
  \Rm
  = \check{\Rm} + u\hat{\Rm}
  - \frac12 \big|\cv(\log u^{1/2})\big|^2 (u \hat g)\KN (u \hat g)
  -2\, u\hat g \KN \big( u^{-1/2} \check\nabla \nabla u^{1/2}\big).
\end{equation}
\medskip

We now analyze the curvatures of multiply-warped products of the form~\eqref{eq:warped-product} on a
manifold $\mc M = \mc B\times\mc F_1\times\cdots\times\mc F_A$. As above, given a function $f:\mc B\to\mb R$, 
there is a natural function $\tilde f:\mc M\to\mb R$ defined by $\tilde f(x,\,y_1,\dots,y_A)=f(x)$.

\begin{claim}\label{claim:hessfmultiple}
If $f$ and $\tilde f$ are as above, then
\[
\nabla \nabla \tilde f = \check\nabla\nabla f + \sum_{\alpha = 1}^A
\big\lp\nabla (\log u_\alpha^{1/2}),\cv f\big\rp\, (u_\alpha g_{\mc F_\alpha}).
\]
\end{claim}
\begin{proof}
  This follows by induction on the number of fibers in the multiply-warped product, using Claim \ref{claim:hessf} as the base case.
  In the induction step, we regard the multiply-warped product with $A$ fibers as a singly-warped product over a base that
  is a multiply-warped product with $A-1$ fibers.
\end{proof}

Our next result provides the curvature formulas we need for this paper. We show below that it also leads directly to
estimate~\eqref{eq:curv-near-cyl}. In stating it, we write formula~\eqref{eq:rmmultiple} in terms of
$u_\alpha^2\Rm[g_{\mc F_\alpha}]$ and $u_\alpha g_{\mc F_\alpha}$ because, for fixed $g_{\mc F_\alpha}$,
these have constant norms with respect to $g$ if we vary $u_\alpha$. 

\begin{lemma}\label{claim:rmmultiple}
The $(4,0)$-tensor $\Rm$ of the metric~\eqref{eq:warped-product} on the multiply-warped product $\mc M$ is given by
  \begin{subequations}\label{eq:rmmultiple}
  \begin{align}
  \Rm[g]
  &= \Rm[g_{\mc B}] +\sum_{\alpha=1}^A u_\alpha^{-1}\,\big(u_\alpha^2 \Rm[g_{\mc F_\alpha}]\big)	\label{0-deriv}\\
  &\quad-\frac12 \sum_{\alpha = 1}^A |\cv(\log u_\alpha^{1/2})|^2\,\big(u_\alpha g_{\mc F_\alpha}\KN u_\alpha g_{\mc F_\alpha}\big)
  	\label{eq:rmmultiple-fiber} \\
  &\quad- \sum_{\alpha = 1}^{A} \sum_{\beta = 1}^{\alpha -1} \Big\lp\nabla (\log u_\alpha^{1/2}),\nabla (\log u_\beta^{1/2})\Big\rp
  \big(u_\alpha g_{\mc F_\alpha} \KN u_\beta g_{\mc F_\beta}\big) \\
  &\quad- 2\sum_{\alpha = 1}^A u_\alpha\, g_{\mc F_\alpha} \KN \big(u_\alpha^{-1/2}\cv_{g_{\mc B}}\cv (u_\alpha^{1/2})\big).
  \end{align}
  \end{subequations}
\end{lemma}

\begin{proof}
  This follows by an induction argument similar to that in Claim~\ref{claim:hessfmultiple}.
  The induction hypothesis is that the claim holds for a metric with $A-1$ fibers, which we denote by
  $g_{(A-1)}:=g_{\mc B}+\sum_{\alpha=1}^{A-1} u_\alpha g_{\mc F_\alpha}$.
  We denote the curvature and connection of $g_{(A-1)}$ by $\Rm_{(A-1)}$ and $\nabla_{(A-1)}$.
  
  We may apply formula~\eqref{eq:rm_single} for the curvature of a singly warped product to write the curvature of
  $g = g_{(A-1)} +  u_A g_{\mc F_A}$ in terms of $\Rm_{(A-1)}$, obtaining
  \begin{align*}
  \Rm[g]
  &= \Rm_{(A-1)} + u_A\Rm[g_{\mc F_A}]\\
  &\quad- \frac12 \big|\cv(\log u_A^{1/2})\big|^2 (u_A g_{\mc F_A})\KN (u_A g_{\mc F_A})\\
 &\quad -2\, u_A g_{\mc F_A} \KN \big( u_A^{-1/2} \nabla_{(A-1)} \nabla u_A^{1/2}\big).
  \end{align*}
 Using Claim~\ref{claim:hessfmultiple}, we rewrite the Hessian term in the last line above as
  \[
    u_A^{-1/2}\nabla_{(A-1)}\nabla u_A^{1/2}
    = u_A^{-1/2}\nabla_{g_{\mc B}} \nabla u_A^{1/2}
      + \sum_{\beta = 1}^{A-1}
      \big\lp \nabla \log u_\beta^{1/2}, \nabla \log u_A^{1/2}\big\rp (u_{\beta} g_{\mc F_\beta}).
  \]
  This completes the induction step.
  
  In summary, this induction argument shows that adding an  additional fiber to a multiply-warped product adds an
  additional term to each (outer) sum in~\eqref{eq:rmmultiple}.
\end{proof}

\begin{remark}		\label{rem-curv}
It follows easily from Lemma~\ref{claim:rmmultiple} that
there exists a universal constant $C$ depending only on the dimensions such that
\[
|\Rm|_g \leq |\Rm[g_{\mc B}]|_{g_{\mc B}}
+C\sum_{\alpha=1}^A
\Big(
  u_\alpha^{-1}|u_\alpha^2\Rm[g_{\mc F_\alpha}]|_g+u_\alpha^{-2}|\cv v_\alpha|^2_g
  +u_\alpha^{-1}|\check\cv\cv v_\alpha|_{g_{\mc B}}
\Big).
\]
Furthermore, one sees readily that
\begin{align*}
\left|\Rm[g]-\sum_{\alpha=1}^A u_\alpha^{-1}\Rm[g_{\mc F_\alpha}]\right|_g
	&\leq C\left\{\rho^{1/2}+\sum_{\alpha=1}^A\left(u_\alpha^{-2}\gamma_\alpha+u_\alpha^{-1}\chi_\alpha^{1/2}\right)\right\},
\end{align*}
where $\rho,\gamma_\alpha,\chi_\alpha$ are defined in~\eqref{eq-def-quantity}. This is estimate~\eqref{eq:curv-near-cyl}.
\end{remark}
\medskip

To conclude, we calculate the components of the Ricci tensor. We obtain
\begin{align}
R_{ij}&=g^{JK}R_{iJKj}\notag \\
      &=\check R_{ij}
        -\sum_{\alpha=1}^A n_\alpha
        \left(
        \frac12 u_\alpha^{-1}\cv_i\cv_j u_\alpha
        -\frac14 u_\alpha^{-2}\cv_i u_\alpha\cv_j u_\alpha
        \right),
\label{eq:Rc-horizontal}
\end{align}
and on each fiber $\mc F_\alpha$,
\begin{align*}
R_{\tau\nu}&=R_{i\tau\nu}^i + R_{\sigma\tau\nu}^\sigma\\
&=(\hat R_\alpha)_{\tau\nu}-
\left(\frac12 \Delta_{\mc B} u_\alpha  -\frac12 u_\alpha^{-1}|\cv u_\alpha|^2 + \frac12 \sum_{\beta = 1}^An_\beta \innerp{u_\alpha}{\cv \log u_\beta^{1/2}}\right)(\hat g_\alpha)_{\tau\nu}.
\end{align*}
In the last formula, the Laplacian on the \textsc{rhs} is computed with respect to the metric $g_{\mc B}$ on the base.
To match the convention used elsewhere in this paper, we rewrite the expression in terms of the Laplacian
$\Delta\equiv\Delta_{\mc M}$ computed with respect to the metric $g$ on the total space $\mc M$.
Using~\eqref{eq:compare-Laplacians}, we obtain
\begin{equation}	\label{eq:Rc-vertical}
R_{\tau\nu}=(\hat R_\alpha)_{\tau\nu}-
\frac12\Big(\Delta u_\alpha -u_\alpha^{-1}|\cv u_\alpha|^2\Big)(\hat g_\alpha)_{\tau\nu}.
\end{equation}
Formulas~\eqref{eq:Rc-horizontal} and~\eqref{eq:Rc-vertical} directly imply the system~\eqref{eq:Ricci-flow-system}
of evolution equations that results if one evolves the metric $g$ on the total space by Ricci flow.

\section{Laplacians of tensor seminorms}

For use in Appendix~\ref{evolve-curvatures} below, we here compute and estimate the Laplacians
of various tensor seminorms. We continue the conventions of Appendix~\ref{multiply-warped}, using
lowercase Roman indices (\emph{e.g.,} $i,j,k,\ell$) for horizontal vectors, lowercase Greek indices
(\emph{e.g.,} $\sigma,\tau,\nu,\omega$) for vertical vectors, and allowing capital Roman letters to
range over both sets of indices. We  continue denoting $g_{\mc B}$ by $\check g$ when working in local coordinates.

Before treating Laplacians of seminorms, we establish some preliminary results for first derivatives of tensor fields.

\begin{claim}	\label{lem:Tim}
If $T$ is an $(m,0)$-tensor field such that $T(U_1,U_2,\dots,U_m)$ vanishes if exactly
one $U_k$ is vertical, then
\[
\cv T\Big|_{T\mc M\otimes(T\mc B)^m} = \cv_{g_{\mc B}}\Big(T\Big|_{(T\mc B)^m}\Big).
\]
\end{claim}

\begin{proof}
In the proof, we denote horizontal vector fields by $H_1,H_2,\dots$ and vertical vector fields by $V, V'$.
For simplicity, we illustrate the idea of the proof with $m=3$. The generalization to arbitrary $m$ is clear.
The key fact is that the only components of the
connection in~\eqref{eq:Christoffel} that differ from those of a direct (\emph{i.e.,} non-warped) product
are those that exchange horizontal and vertical vectors. Specifically, we have
\begin{align*}
\cv T(V,H_1,H_2,H_3)=-V^\sigma\Gamma_{\sigma\ell}^\tau\big(
H_1^\ell H_2^j H_3^k T_{\tau jk}+H_1^i H_2^\ell H_3^k T_{i\tau k}
+H_1^i H_2^j H_3^\ell T_{ij\tau}\big)=0.
\end{align*}
Hence $\cv T(U,H_1,H_2,H_3)$ can be nonzero only if $U=H_4$ is horizontal.

We note that the assumption that $H_1,H_2,H_3$ are horizontal is necessary: indeed,
similar reasoning shows that terms like $\cv T(V,H_1,H_2,V')$ are nonzero in general.
\end{proof}

\begin{claim}	\label{nabla-two-tensor}
If $T$ is a symmetric $(2,0)$-tensor field with no nonzero horizontal-vertical components, then all components
of $\cv T$ for a warped product are the same as those for a direct product (\emph{i.e.,} a metric
with $u$ constant) except
\begin{align*}
\cv_i T_{\sigma\tau}&=-u^{-1}\cv_i u\,T_{\sigma\tau},\\
\cv_\sigma T_{i\tau}&=\cv_\sigma T_{\tau i}
	=\frac12 u^{-1} \cv^j u\,T_{ji}\,g_{\sigma\tau}-\frac12u^{-1}\cv_i u\,T_{\sigma\tau},
\end{align*}
which do not vanish in general.
\end{claim}

\begin{proof}
Direct computation using~\eqref{eq:Christoffel}.
\end{proof}

We note for use below that Claim~\ref{nabla-two-tensor} implies 
easily that all components of $\cv g_{\mc B}$ vanish identically except
\begin{equation}	\label{eq:grad-hat-g}
\cv_\sigma\check g_{i\tau}= \cv_\sigma\check g_{\tau i}=\frac12 u^{-1}\cv_i u\,g_{\sigma\tau}.
\end{equation}
\smallskip

For clarity, before deriving an estimate for multiply-warped products, we first perform an exact calculation for a
singly-warped product.
We continue to assume that $T$ is a symmetric $(2,0)$-tensor field with no nonzero horizontal-vertical components.
Then we have
\begin{align}
\cv_Q|T|^2_{g_{\mc B}} &= 2(\cv_Q\check g^{IK})\check g^{JL}T_{IJ}T_{KL}
	+2\check g^{IK}\check g^{JL}(\cv_Q T_{IJ})T_{KL}	\notag\\
&=2\check g^{IK}\check g^{JL}(\cv_Q T_{IJ})T_{KL},		\label{eq:one-derivative-T}
\end{align}
because
\[
(\cv_P\check g^{IK})\check g^{JL}T_{IJ}T_{KL}=
\cv_\sigma\check g^{i\tau}\check g^{j\ell}T_{ij}T_{\tau\ell}=0
\]
by assumption. Thus we obtain
\begin{align*}
\Delta|T|^2_{g_{\mc B}} &= 2g^{PQ}\cv_P\big\{\check g^{IK}\check g^{JL}(\cv_Q T_{IJ})T_{KL}\big\}\\
&=2g^{PQ}\check g^{IK}\check g^{JL}(\cv_P\cv_Q T_{IJ})T_{KL}
	+2g^{PQ}\check g^{IK}\check g^{JL}(\cv_P T_{IJ})(\cv_Q T_{KL})\\
	&\qquad+4g^{PQ}(\cv_P\check g^{IK})\check g^{JL}(\cv_Q T_{IJ})T_{KL}.
\end{align*}
Writing this invariantly, we have
\[
\Delta|T|^2_{g_{\mc B}} = 2\lp\Delta T,T\rp_{g_{\mc B}} + 2|\cv T|^2_{g_{\mc B}} + 4\mc Z[T],
\]
where
\begin{align*}
\mc Z[T]&:=g^{PQ}(\cv_P\check g^{IK})\check g^{JL}(\cv_Q T_{IJ})T_{KL}\\
&=g^{\sigma Q}(\cv_\sigma\check g^{i\tau})\check g^{j\ell}(\cv_Q T_{IJ})T_{\tau\ell}
	+g^{\sigma Q}(\cv_\sigma\check g^{\nu k})\check g^{j\ell}(\cv_Q T_{\nu j})T_{k\ell}\\
	&=g^{\sigma Q}(\cv_\sigma\check g^{\nu k})\check g^{j\ell}(\cv_Q T_{\nu j})T_{k\ell}\\
	&=-\frac12 u^{-1}\check g^{j\ell}(\cv^k uT_{k\ell})(\cv^\tau T_{\tau j}).
\end{align*}
Note that we use~\eqref{eq:grad-hat-g} in the final step. We expand the divergence factor, obtaining
\[
\cv^\tau T_{\tau j}=\frac{\dim(\mc F)}{2}u^{-1}\cv^i u T_{ij}-\frac12u^{-1}\cv_j u (\hat{\mr{tr}}T),
\]
where $\hat{\mr{tr}}T:=g^{\sigma\tau}T_{\sigma\tau}$ denotes the trace of the vertical components of $T$.
Combining factors, we write $\mc Z[T]$ invariantly as
\[
\mc Z[T] = \frac14 u^{-2} (\hat{\mr{tr}}T) \lp T,\cv u \otimes\cv u\rp_{g_{\mc B}}
	-\frac{\dim(\mc F)}{4}u^{-2}|T(\cv u)|^2_{g_{\mc B}},
\]
where in the second term, we regard $T$ as an endomorphism of the tangent bundle. This work proves:

\begin{lemma}	\label{lem:norm-Laplacian}
If $T$ is a symmetric $(2,0)$-tensor field with no nonzero horizontal-vertical components on a warped product, then
\begin{align*}
-\Delta|T|^2_{g_{\mc B}}&=-2\lp\Delta T,T\rp_{g_{\mc B}}-2|\cv T|^2_{g_{\mc B}}\\
&\quad+\dim(\mc F)\, u^{-2}|T(\cv u)|^2_{g_{\mc B}} -u^{-2} (\hat{\mr{tr}}T) \lp T,\cv u \otimes\cv u\rp_{g_{\mc B}}.
\end{align*}
\end{lemma}

Generalizing this to the multiply-warped products we study in this paper, one readily obtains:

\begin{corollary}	\label{estimate-norm-Laplacian}
If $T$ is a symmetric $(2,0)$-tensor field with no nonzero horizontal-vertical components, then there exists
a constant $C$ depending only on the dimension vector $\vec N=(n,n_\alpha)$ such that
\[
-\Delta|T|^2_{g_{\mc B}}\leq-2\lp\Delta T,T\rp_{g_{\mc B}}-2|\cv T|^2_{g_{\mc B}}
+C \Big(\sum_{\alpha=1}^A |\cv \log u_\alpha|^2\Big) |T| |T|_{g_{\mc B}}.
\]
\end{corollary}
\bigskip

We now proceed to estimate $-\Delta|\Rm|_{g_{\mc B}}^2$ on a multiply-warped product.
Because the details are so similar to the previous case, we merely sketch the proof.
First, Claim~\ref{lem:Tim} shows that $\cv\Rm$ vanishes if exactly one index
is vertical. Thus we see by~\eqref{eq:grad-hat-g} that
\[
\cv_S |\Rm|_{g_{\mc B}}^2 = 2\check g^{IW}\check g^{JX}\check g^{KY}\check g^{LZ}
	(\cv_S R_{IJKL})R_{WXYZ},
\]
exactly as in~\eqref{eq:one-derivative-T}. Thus we find that
\[
\Delta |\Rm|_{g_{\mc B}}^2=2\lp\Delta\Rm,\Rm\rp_{g_{\mc b}}+2|\cv\Rm|^2_{g_{\mc B}}+8\,\mc Z[\Rm],
\]
where
\begin{equation}	\label{eq:Rm-quadratic}
\mc Z[\Rm]:=g^{\nu\sigma}(\cv_\nu\check g^{\tau w})\check g^{jx}\check g^{ky}\check g^{\ell z}
	(\cv_\sigma  R_{\tau jk\ell})R_{wxyz}.
\end{equation}

\begin{claim} \label{lem:fiber_derivative}
The $(5,0)$-tensor field $\cv\Rm$ satisfies
\[
  \nabla_\sigma R_{\tau jk\ell}
  =- \Gamma_{\sigma \tau}^i R_{ijk\ell} + \Gamma_{\sigma k}^\nu R_{\tau j \ell \nu}
  	- \Gamma_{\sigma \ell}^\nu R_{\tau j k \nu}.
\]
\end{claim}

\begin{proof}
Using equations~\eqref{eq:Christoffel}, \eqref{eq:null-plane}, and the fact that $R_{\tau jk\ell} = 0$, we compute that
  \begin{align*}
    \nabla_\sigma R_{\tau jk\ell}
    &= - \Gamma_{\sigma \tau}^I R_{Ijk\ell} - \Gamma_{\sigma j}^I R_{\tau I k\ell}
    - \Gamma_{\sigma k}^I R_{\tau j I \ell} - \Gamma_{\sigma\ell}^I R_{\tau j k I} \\
    &= - \Gamma_{\sigma \tau}^i R_{ijk\ell} - \Gamma_{\sigma j}^\nu R_{\tau \nu k\ell}
    - \Gamma_{\sigma k}^\nu R_{\tau j \nu\ell} - \Gamma_{\sigma \ell}^\nu R_{\tau j k \upsilon}  \\
    &= - \Gamma_{\sigma \tau}^i R_{ijk\ell} + \Gamma_{\sigma k}^\nu R_{\tau j \ell \nu}
    - \Gamma_{\sigma\ell}^\nu R_{\tau j k \nu}. 
  \end{align*}
\end{proof}

We denote by $\mc H$ the (integrable) horizontal distribution of  $\mc M$ and by
$\Rm_{\mc H\otimes\mc H}$ the restriction
\[
\Rm_{\mc H\otimes\mc H}:=\Rm \big|_{\mc H \otimes T \mc M \otimes T \mc M \otimes \mc H},
\]
\emph{i.e.,} only those components of $\Rm$ having the form $R_{iJK\ell}$. Then equation~\eqref{eq:grad-hat-g},
equation~\eqref{eq:Rm-quadratic}, and Claim~\ref{lem:fiber_derivative} immediately imply the following:

\begin{corollary}	\label{estimate-Laplacian|Rm|}
There exists a constant $C$ depending only on the dimension vector $\vec N=(n,n_\alpha)$ such that
  \begin{align*}
    - \Delta |\Rm|_{g_{\mc B}}^2
    &\leq
    - 2 \lp\Delta \Rm, \Rm\rp_{g_{\mc B}}
    - 2 |\nabla \Rm|^2_{g_{\mc B}}\\
    &\quad+ C \Big(\sum_{\alpha=1}^A |\cv \log u_\alpha|^2\Big) |\Rm|_{g_{\mc B}} |\Rm_{\mc H\otimes\mc H}|_g.
  \end{align*} 
\end{corollary}

\section{Curvature evolution equations and estimates}	\label{evolve-curvatures}

We continue the convention of Appendix~\ref{multiply-warped}, using lowercase Roman indices (\emph{e.g.,} $i,j,k,\ell$)
for horizontal vectors, lowercase Greek indices (\emph{e.g.,} $\sigma,\tau,\nu,\omega$) for vertical vectors, and allowing
capital Roman letters to range over both sets of indices. We assume that the metric $g$ is evolving by the
Ricci flow system~\eqref{eq:Ricci-flow-system}.

\subsection{The evolution of $\rho$}
Under Ricci flow, the $(4,0)$-Riemann curvature tensor evolves by (see, \emph{e.g.,} Corollary~6.14 of~\cite{CK04})
\begin{align*}
\heat R_{IJKL}&=g^{PQ}\big(R_{IJP}^M R_{MQKL}-2R^M_{PIK} R_{JQML}+2R_{PIML}R^M_{JQK}\big)\\
&\quad -(R_I^P R_{PJKL}+R_J^P R_{IPKL} + R_K^P R_{IJPL} + R_L^P R_{IJKP}).
\end{align*}

For simplicity, we again begin with an exact calculation for a singly-warped product and generalize this
below to an estimate for multiply-warped products. 
We start by computing the evolution of the curvature tensor acting on horizontal vectors, finding that
\begin{subequations}
\begin{align}
\heat R_{ijk\ell} &= g^{ab}(R_{ija}^c R_{cbk\ell}-2R_{aik}^cR_{jbc\ell}+2R_{aic\ell}R_{jbk}^c)	\label{eq:horizontal}\\
&\quad +g^{\sigma\tau}(R_{ij\sigma}^\gamma R_{\gamma\tau kl} -2 g_{\ell m}R_{\sigma ik}^\gamma R_{j\tau \gamma}^m
+2R_{\sigma i\gamma\ell}R_{j\tau k}^\gamma)	\label{eq:horizontal-evolution-first}\\
&\quad -(R_i^P R_{Pjk\ell}+R_j^P R_{iPk\ell} + R_k^P R_{ijP\ell} + R_\ell^P R_{ijkP}).	\label{eq:skip}
\end{align}
\end{subequations}
We note that~\eqref{eq:horizontal} consists of the terms one would see if the base alone were evolving by
Ricci flow, while~\eqref{eq:skip} consists of terms that are cancelled by derivatives of $g^{-1}$ in our
calculation of the evolution of $\rho=|\Rm|^2_{g_{\mb B}}$ below. So we need only to examine the three additional
terms in~\eqref{eq:horizontal-evolution-first}.

By~\eqref{eq:null-plane}, the first term in~\eqref{eq:horizontal-evolution-first} vanishes.
To evaluate the second and third terms in~\eqref{eq:horizontal-evolution-first}, we can
apply the formulas derived in Appendix~\ref{multiply-warped} directly, obtaining
\begin{align*}
g_{\ell m}R_{\sigma ik}^\gamma R_{j\tau \gamma}^m
&=\hat g_{\sigma\tau}\Big(\frac14 u^{-1}\cv_i\cv_k u \cv_j\cv_\ell u
-\frac18 u^{-2}\cv_i\cv_k u\cv_j u\cv_\ell u\\
&\qquad\qquad-\frac18 u^{-2}\cv_j\cv_\ell u \cv_i u\cv_k u
+\frac{1}{16}u^{-3}\cv_i u\cv_j u\cv_k u\cv_\ell u\Big)
\end{align*}
and
\begin{align*}
R_{\sigma i\gamma\ell}R_{j\tau k}^\gamma&=g_{\gamma\zeta}R_{\sigma i\ell}^\zeta R_{\tau jk}^\gamma\\
&=\hat g_{\sigma\tau}\Big(\frac14 u^{-1}\cv_i\cv_\ell u \cv_j \cv_k u
-\frac18 u^{-2}\cv_i\cv_\ell u \cv_j u\cv_k u\\
&\qquad\qquad-\frac18 u^{-2}\cv_j\cv_k u \cv_i u \cv_\ell u
+\frac{1}{16}u^{-3}\cv_i u \cv_j u \cv_k u \cv_\ell u\Big).
\end{align*}

Combining terms and tracing by $g^{\sigma\tau}$, we conclude that
\begin{equation}	\label{eq:horizontal-evolution-last}
\begin{split}
\heat R_{ijk\ell} &= g^{ab}(R_{ija}^c R_{cbk\ell}-2R_{aik}^cR_{jbc\ell}+2R_{aic\ell}R_{jbk}^c)\\
&\quad+\frac{\dim(\mc F)}{2}\Big\{
u^{-2}\big(\cv_i\cv_\ell u\cv_j\cv_k u - \cv_i\cv_k u \cv_j\cv_\ell u\big)\\
&\qquad\qquad\qquad+\frac12 u^{-3}
\big(\cv_i\cv_k u \cv_j u \cv_\ell u+\cv_j\cv_\ell u \cv_i u \cv_k u\\
&\qquad\qquad\qquad\qquad\qquad-\cv_j\cv_k u \cv_i u \cv_\ell u -\cv_i \cv_\ell u \cv_j u \cv_k u\big)\Big\}\\
&\quad -(R_i^P R_{Pjk\ell}+R_j^P R_{iPk\ell} + R_k^P R_{ijP\ell} + R_\ell^P R_{ijkP}).
\end{split}
\end{equation}
\medskip

We now estimate the evolution of $\rho(x,t)=\big|\Rm(x,t)\big|^2_{g_{\mc B}}$ for a multiply-warped product.
We note that in the case of a multiply-warped product, the only
possible nonzero terms in~\eqref{eq:horizontal-evolution-first} occur where the vertical coordinates $\sigma$ and $\tau$ are tangent to the same fiber.
Thus we obtain a sum of derivatives of $u_\alpha$ in~\eqref{eq:horizontal-evolution-last},
and using our estimate derived in Corollary~\ref{estimate-Laplacian|Rm|}, we recover the
standard estimate for the evolution of the curvature norm (see, \emph{e.g.,} Lemma 7.4 of~\cite{CK04})
modified by additional terms coming from the warped-product structure, namely
\begin{equation}	\label{eq:evolve-rho}
\begin{split}
\heat\rho &\leq-2|\cv\Rm|^2_{g_{\mc B}}+C_n\rho^{3/2}\\
&\quad+2\sum_{\alpha=1}^A n_\alpha\Big\{
u_\alpha^{-2}\Rm_{\mc B}(\cv^2v_\alpha,\cv^2v_\alpha)\\
&\qquad\qquad\qquad-2u_\alpha^{-3}\Rm_{\mc B}(\cv^2v_\alpha,\cv v_\alpha\otimes\cv v_\alpha)\Big\}\\
&\quad+C \Big(\sum_{\alpha=1}^A |\cv \log u_\alpha|^2\Big) |\Rm|_{g_{\mc B}} |\Rm_{\mc H\otimes\mc H}|_g,
\end{split}
\end{equation}
where $n_\alpha=\dim(\mc F_\alpha)$, $\Rm_{\mc B}$ denotes the curvature tensor of $g_{\mc B}$,
and $C$ is a constant depending only on the dimension vector $\vec N=(n,n_\alpha)$.

\subsection{The evolution of $\gamma_\alpha$}
We next consider the evolution of the curvature tensor acting on vertical vectors in an arbitrary fiber $\mc F_\alpha$.
It follows from~\eqref{0-deriv} and~\eqref{eq:rmmultiple-fiber} that for a multiply-warped product with space-form fibers,
it suffices to calculate the evolution of $\gamma_\alpha =|\cv u_\alpha|^2 = |\nabla v_\alpha|^2$.

As elsewhere in this Appendix, we omit the fiber index for convenience in the computations below.
We note that $\heat u$ is given by~\eqref{eq:evolve-fiber}.
It also follows from~\eqref{eq:evolve-fiber} that
\[
\partial_t \gamma=2\big\{\lp\cv\Delta v,\cv v\rp - u^{-1}\lp\cv\gamma,\cv v\rp
+u^{-2}\gamma^2\big\} +2\Rc(\cv v,\cv v),
\]
where $\Rc$ denotes the Ricci tensor of $g$ acting on horizontal vectors, as in~\eqref{eq:Rc-horizontal}.
Recalling that $\Delta\gamma=2\lp\Delta\cv v,\cv v\rp+2|\cv\cv v|^2$, we
commute covariant derivatives and conclude that
\begin{equation}	\label{eq:gamma-evolution-derived}
\heat\gamma=-2|\cv\cv v|^2 - 2u^{-1}\lp\cv\gamma,\cv v\rp+2u^{-2}\gamma^2.
\end{equation}
Observing that $\lp\cv\gamma, \cv v\rp = 2\cv^2 v(\cv v, \cv v)$,
we obtain the formula used in Lemma~\ref{lemma-ev-eq}.

\subsection{The evolution of $\chi_\alpha$}
We move on to controlling $\chi_\alpha = |\nabla \nabla v_\alpha|_{g_\mc B}$.  By Remark~\ref{rem-curv}, this is the last quantity needed
to control the full curvature tensor. For simplicity, we again fix a fiber and omit subscripts.

We denote the heat operator with the Lichnerowicz Laplacian of the metric $g$ by $\heat_{\mc L}$. Using the standard formula
(see, \emph{e.g.}, Lemma 2.33 of~\cite{CLN06})
\[
\heat_{\mc L}\cv_I\cv_J v = \cv_I \cv_J \heat v,
\]
we compute this heat operator acting on the covariant Hessian of $v$ as follows:
\begin{align*}
\Big(\heat_{\mc L}\big(\cv^2 v\big)\Big)_{IJ} &= u^{-2}(\cv_I\cv_J v)\,\gamma-2u^{-3}(\cv_I v\cv_J v)\,\gamma\\
&\quad+u^{-2}(\cv_I v\cv_J \gamma+\cv_I \gamma \cv_J v)-u^{-1}\cv_I\cv_J\gamma.
\end{align*}

Now using the identity $-\Delta = -\Delta_{\mc L} + 2\Rm* - 2\Rc*$, where $\Rm$ and $\Rc$
are those of the metric $g$, we convert this formula to one using the standard heat operator:
\begin{align*}
\Big(\heat\cv^2 v\Big)_{ij}&=u^{-2}(\cv_i\cv_j v)\gamma-2u^{-3}(\cv_i v\cv_j v)\gamma\\
&\quad+u^{-2}\big(\cv_i v\cv_j\gamma+\cv_i\gamma\cv_j v\big)-u^{-1}\cv_i\cv_j\gamma\\
&\quad+2R_{ik\ell j}\cv^k\cv^\ell v-R_i^k\cv_k\cv_j v-R_j^k\cv_i\cv_k v\\
&\quad+Nu^{-2}\gamma\big(-\frac12\cv_i\cv_j v+\frac14\cv_iv\cv_jv\big),
\end{align*}
where $N:=\sum_{\beta=1}^A\dim\mc F_\beta$ is the total dimension of the fibers.
We obtain the last line above by simplifying $2R_{i\sigma\tau j}\cv^\sigma\cv^\tau v$ using the identities
\[
R_{i\sigma\tau j}=u^{-1}g_{\sigma\tau}\big(-\frac12\cv_i\cv_j v + \frac14 u^{-1}\cv_i v\cv_j v\big)\quad\mbox{ and }\quad
\cv^\sigma\cv^\tau v = \frac12 u^{-1}\gamma g^{\sigma\tau}.
\]
Finally, we apply Corollary~\ref{estimate-norm-Laplacian} to conclude that
\begin{equation}	\label{eq:Hessian-evolution}
\begin{split}
\heat\chi&\leq-2|\cv^3 v|_{g_{\mc B}}^2+4\Rm_{\mc B}(\cv^2 v, \cv^2 v)+2u^{-2}\gamma\chi\\
&\quad-2u^{-3}\lp\cv v,\cv\gamma\rp\gamma+4u^{-2}\lp\cv^2 v,\cv v\otimes\cv\gamma\rp\\
&\quad-2u^{-1}\lp\cv^2 v,\cv^2 \gamma\rp_{g_{\mc B}}
+Nu^{-2}\gamma\big\{-\chi+\frac14u^{-1}\lp\cv\ v,\cv\gamma\rp\big\}\\
&\quad+C \Big(\sum_{\alpha=1}^A |\cv \log u_\alpha|^2\Big) |\cv^2v| |\cv^2 v|_{g_{\mc B}},
\end{split}
\end{equation}
where $\Rm_{\mc B}$ again denotes the curvature tensor of $g_{\mc B}$, and $C$ is a constant depending
only on the dimension vector $\vec N=(n,n_\alpha)$.

\end{document}